\documentclass[final,leqno]{siamltex}
\usepackage{amssymb,amsmath,bbm}
\usepackage[usenames,dvipsnames]{color}
\usepackage{mathrsfs}
\usepackage{hyperref}
\usepackage[all]{xy}
\usepackage{graphicx,xspace,bm,subfigure}
\usepackage{url}
\usepackage[noend,lined,linesnumbered,algoruled]{algorithm2e}
\usepackage{setspace}

\newtheorem{remark}[theorem]{Remark}
\newtheorem{example}[theorem]{Example}

\newcommand{\CC}{\mathcal{C}}
\renewcommand{\SS}{\mathcal{S}}
\newcommand{\TT}{\mathcal{T}}
\newcommand{\WW}{\mathcal{W}}

\newcommand{\DD}{\mathcal{D}}

\newcommand{\NN}{\mathcal{N}}

\newcommand{\UU}{\mathcal{U}}
\renewcommand{\AA}{\mathcal{A}}
\newcommand{\MM}{\mathcal{M}}
\newcommand{\EE}{\mathcal{E}}

\newcommand{\II}{\mathcal{I}}
\newcommand{\abs}[1]{\ensuremath{\left\lvert{#1}\right\rvert}}
\newcommand{\norm}[1]{\ensuremath{\| #1 \|}}

\newcommand{\nll}{\mathrm{null}}
\newcommand{\rge}{\mathrm{range}}
\newcommand{\realpart}{\mathrm{Re}}

\newcommand{\real}{{\mathbb{R}}}
\newcommand{\realpositive}{{\mathbb{R}}_{>0}}
\newcommand{\realnonnegative}{{\mathbb{R}}_{\ge 0}}
\newcommand{\realnonpositive}{{\mathbb{R}}_{\le 0}}

\newcommand{\integerspositive}{\mathbb{Z}_{\geq 1}}

\newcommand{\eps}{\epsilon}

\newcommand{\until}[1]{\{1,\dots,#1\}}
\newcommand{\map}[3]{#1:#2 \rightarrow #3}
\newcommand{\setmap}[3]{#1:#2 \rightrightarrows #3}
\newcommand{\setr}[1]{\{#1\}}

\newcommand{\Lie}{\mathcal{L}}

\newcommand{\gradient}{\nabla}

\newcommand{\setdef}[2]{\{#1 \; | \; #2\}}

\newcommand{\rrarrows}{\rightrightarrows}

\newcommand{\levelset}[2]{#1^{-1}({\leq #2})}

\newcommand{\saddleset}[1]{\operatorname{Saddle}(#1)}

\newcommand{\xo}{x_{*}}
\newcommand{\zo}{z_{*}}
\newcommand{\xb}{\bar{x}}
\newcommand{\zb}{\bar{z}}
\newcommand{\xt}{\tilde{x}}
\newcommand{\zt}{\tilde{z}}
\newcommand{\SD}{X_{\text{sp}}}

\newcommand{\proj}{\mathrm{proj}}

\newcommand{\oprocendsymbol}{\hbox{$\bullet$}}
\newcommand{\oprocend}{\relax\ifmmode\else\unskip\hfill\fi\oprocendsymbol}

\newcommand{\longthmtitle}[1]{\mbox{}\textup{\textsl{(#1):}}}

\renewcommand{\theenumi}{(\roman{enumi})}

\allowdisplaybreaks

\begin{document}

\title{Saddle-point dynamics: conditions for asymptotic stability of
  saddle points \thanks{A preliminary version of this paper appeared at
    the 2015 American Control Conference as~\cite{AC-JC:15-acc}.}}

  \author{Ashish Cherukuri\footnotemark[2] \and 
         Bahman Gharesifard\footnotemark[3] \and
         Jorge Cort\'{e}s\footnotemark[2]}
\renewcommand{\thefootnote}{\fnsymbol{footnote}}
  \footnotetext[2]{Department of Mechanical and Aerospace Engineering, University of California, San Diego, 9500 Gilman Dr., La Jolla, CA 92093-0411, United States, \texttt{\{acheruku,cortes\}@ucsd.edu}}
  \footnotetext[3]{Department of Mathematics and Statistics, Queen's University, 403 Jeffery Hall, University Ave., Kingston, ON, Canada K7L3N6, \texttt{bahman@mast.queensu.ca}}

\maketitle

\begin{abstract}
  This paper considers continuously differentiable functions of two
  vector variables that have (possibly a continuum of) min-max saddle
  points. We study the asymptotic convergence properties of the
  associated saddle-point dynamics (gradient-descent in the first
  variable and gradient-ascent in the second one).  We identify a
  suite of complementary conditions under which the set of saddle
  points is asymptotically stable under the saddle-point dynamics.
  Our first set of results is based on the convexity-concavity of the
  function defining the saddle-point dynamics to establish the
  convergence guarantees.  For functions that do not enjoy this
  feature, our second set of results relies on properties of the
  linearization of the dynamics, the function along the proximal
  normals to the saddle set, and the linearity of the function in one
  variable.  We also provide global versions of the asymptotic
  convergence results.  Various examples illustrate our discussion.
\end{abstract}
 


\section{Introduction}\label{se:Intro}

It is well known that the trajectories of the gradient dynamics of a
continuously differentiable function with bounded sublevel sets
converge asymptotically to its set of critical points, see
e.g.~\cite{WMH-SS:74}.  This fact, however, is not true in general for
the saddle-point dynamics (gradient descent in one variable and
gradient ascent in the other) of a continuously differentiable
function of two variables, see e.g.~\cite{KA-LH-HU:58,RD-PAS-RS:58}.
In this paper, our aim is to investigate conditions under which the
above statement is true for the case where the critical points are
min-max saddle points and they possibly form a continuum.  Our
motivation comes from the applications of the saddle-point dynamics
(also known as primal-dual dynamics) to find solutions of equality
constrained optimization problems and Nash equilibria of zero-sum
games.

\subsubsection*{Literature review}

In constrained optimization problems, the pioneering
works~\cite{KA-LH-HU:58,TK:56} popularized the use of the primal-dual
dynamics to arrive at the saddle points of the Lagrangian. For
inequality constrained problems, this dynamics is modified with a
projection operator on the dual variables to preserve their
nonnegativity, which results in a discontinuous vector field.
Recent works have further explored the convergence analysis of such
dynamics, both in continuous~\cite{DF-FP:10,AC-EM-JC:16-scl} and in
discrete~\cite{AN-AO:09-jota} time. The work~\cite{HBD-CE:11} proposes
instead a continuous dynamics whose design builds on first- and
second-order information of the Lagrangian.  In the context of
distributed control and multi-agent systems, an important motivation
to study saddle-point dynamics comes from network optimization
problems where the objective function is an aggregate of each agents'
local objective function and the constraints are given by a set of
conditions that are locally computable at the agent level.  Because of
this structure, the saddle-point dynamics of the Lagrangian for such
problems is inherently amenable to distributed implementation. This
observation explains the emerging body of work that, from this
perspective, looks at problems in distributed convex
optimization~\cite{JW-NE:11,BG-JC:14-tac,GD-ME:14}, distributed linear
programming~\cite{DR-JC:15-tac}, and applications to power
networks~\cite{XM-NE:13,XZ-AP:13,CZ-UT-NL-SL:14} and bargaining
problems~\cite{DR-JC:16-tcns}. The work~\cite{VAK-YSP:87} shows an
interesting application of the saddle-point dynamics to find a common
Lyapunov function for a linear differential inclusion.  In game
theory, it is natural to study the convergence properties of
saddle-point dynamics to find the Nash equilibria of two-person
zero-sum games~\cite{TB-GJO:82,LJR-SAB-SSS:13}.  
A majority of these works assume the function whose saddle points are
sought to be convex-concave in its arguments.  Our focus here instead
is on the asymptotic stability of the min-max saddle points under the
saddle-point dynamics for a wider class of functions, and without any
nonnegativity-preserving projection on individual variables.  We
explicitly allow for the possibility of a continuum of saddle points,
instead of isolated ones, and wherever feasible, on establishing
convergence of the trajectories to a point in the set. The issue of
asymptotic convergence, even in the case of standard gradient systems,
is a delicate one when equilibria are a continuum~\cite{PAA-KK:06}. In
such scenarios, convergence to a point might not be guaranteed, see
e.g., the counter example in~\cite{JP-WD:82}. Our work here is
complementary to~\cite{TH-IL:14}, which focuses on the
characterization of the asymptotic behavior of the saddle-point
dynamics when trajectories do not converge to saddle points and
instead show oscillatory behaviour.

\subsubsection*{Statement of contributions}

Our starting point is the definition of the saddle-point dynamics for
continuously differentiable functions of two (vector) variables, which
we term saddle functions. The saddle-point dynamics consists of
gradient descent of the saddle function in the first variable and
gradient ascent in the second variable. Our objective is to
characterize the asymptotic convergence properties of the saddle-point
dynamics to the set of min-max saddle points of the saddle
function. Assuming this set is nonempty, our contributions can be
understood as a catalog of complementary conditions on the saddle
function that guarantee that the trajectories of the saddle-point
dynamics are proved to converge to the set of saddle points, and possibly
to a point in the set. We broadly divide our results in two
categories, one in which the saddle function has convexity-concavity
properties and the other in which it does not.  For the first
category, our starting result considers saddle functions that are
locally convex-concave on the set of saddle points.  We show that
asymptotic stability of the set of saddle points is guaranteed if
either the convexity or concavity properties are strict, and
convergence is pointwise. Furthermore, motivated by equality
constrained optimization problems, our second result shows that the
same conclusions on convergence hold for functions that depend
linearly on one of its arguments if the strictness requirement is
dropped.  For the third and last result in this category, we relax the
convexity-concavity requirement and establish asymptotic convergence
for strongly jointly quasiconvex-quasiconcave saddle functions.
Moving on to the second category of scenarios, where functions lack
convexity-concavity properties, our first condition is based on
linearization. We consider piecewise twice continuously differentiable
saddle-point dynamics and provide conditions on the eigenvalues of the
limit points of Jacobian matrices of the saddle function at the saddle
points that ensure local asymptotic stability of a manifold of saddle
points. Our convergence analysis is based on a general result of
independent interest on the stability of a manifold of equilibria for
piecewise smooth vector fields that we state and prove using ideas
from center manifold theory. The next two results are motivated by the
observation that saddle functions exist in the second category that do
not satisfy the linearization hypotheses and yet have convergent
dynamics. In one result, we justify convergence by studying the
variation of the function and its Hessian along the proximal normal
directions to the set of saddle points.  Specifically, we assume
polynomial bounds for these variations and derive an appropriate
relationship between these bounds that ensures asymptotic convergence.
In the other result, we assume the saddle function to be linear in one
variable and indefinite in another, where the indefinite part
satisfies some appropriate regularity conditions.  When discussing
each of the above scenarios, we extend the conditions to obtain global
convergence wherever feasible. Our analysis is based on tools and
notions from saddle points, stability analysis of nonlinear systems,
proximal normals, and center manifold theory.  Various 
examples throughout the paper justify the complementary character of
the hypotheses in our results.

\subsubsection*{Organization}
Section~\ref{sec:prelims} introduces notation and basic preliminaries.
Section~\ref{sec:problem} presents the saddle-point dynamics and the
problem statement.  Section~\ref{sec:case1} deals with saddle
functions with convexity-concavity properties. For the case when this
property does not hold, Section~\ref{sec:case2} relies on
linearization techniques, proximal normals, and the linearity
structure of the saddle function to establish convergence guarantees.
Finally, Section~\ref{sec:conclusions} summarizes our conclusions and
ideas for future work.

\section{Preliminaries}\label{sec:prelims}
This section introduces basic notation and presents preliminaries on
proximal calculus and saddle points.

\subsection{Notation}\label{sec:notation}

We let $\real$, $\realnonnegative$, $\realnonpositive$,
$\realpositive$ and $\integerspositive$ be the set of real,
nonnegative real, nonpositive real, positive real, and positive
integer numbers, respectively.  Given two sets $\AA_1, \AA_2 \subset
\real^n$, we let $\AA_1 + \AA_2 = \setdef{x + y}{x \in \AA_1, y \in
  \AA_2}$.  We denote by $\norm{\cdot}$ the $2$-norm on $\real^n$ and
also the induced $2$-norm on $\real^{n \times n}$. Let $B_{\delta}(x)$
represent the open ball centered at $x \in \real^n$ of radius $\delta
> 0$.  Given $x\in \real^n$, $x_i$ denotes the $i$-th component of
$x$.  For vectors $u \in \real^n$ and $w \in \real^m$, the vector
$(u;w) \in \real^{n+m}$ denotes their concatenation.  For $A \in
\real^{n \times n}$, we use $A \succeq 0$, $A \preceq 0$, $A \succ 0$,
and $A \prec 0$ to denote the fact that $A$ is positive semidefinite,
negative semidefinite, positive definite, and negative definite,
respectively. The eigenvalues of $A$ are $\lambda_i(A)$ for $i \in
\until{n}$. If $A$ is symmetric, $\lambda_{\max}(A)$ and
$\lambda_{\min}(A)$ represent the maximum and minimum eigenvalues,
respectively. The range and null spaces of $A$ are denoted by
$\rge(A)$, $\nll(A)$, respectively. We use the notation $\CC^k$ for a
function being $ k \in \integerspositive $ times continuously
differentiable.
A set $\SS \subset \real^n$
is \emph{path connected} if for any two points $a,b \in \SS$ there
exists a continuous map $\map{\gamma}{[0,1]}{\SS}$ such that
$\gamma(0) =a$ and $\gamma(1) = b$. A set $\SS_c \subset \SS \subset
\real^n$ is an \emph{isolated path connected component} of $\SS$ if it
is path connected
and there exists an open neighborhood $\UU$ of $\SS_c$ in $\real^n$
such that $\UU \cap \SS = \SS_c$.  For a real-valued function
$F:\real^n \times \real^m \to \real$, we denote the partial derivative
of $F$ with respect to the first argument by $\gradient_x F$ and with
respect to the second argument by $\gradient_z F$.  The higher-order
derivatives follow the convention $\gradient_{xz} F = \frac{\partial^2
  F}{\partial x \partial z}$, $\gradient_{xx} F = \frac{\partial^{2}
  F}{\partial x^2}$, and so on.  The restriction of
$\map{f}{\real^n}{\real^m}$ to a subset $\SS \subset \real^n$ is
denoted by $f_{|S}$.  The Jacobian of a $\CC^1$ map
$\map{f}{\real^n}{\real^m}$ at $x \in \real^n$ is denoted by $Df(x) \in
\real^{m \times n}$.  For a real-valued function
$\map{V}{\real^n}{\real}$ and $\alpha>0$, we denote the sublevel set
of $V$ by~$\levelset{V}{\alpha} =\setdef{x \in \real^n}{V(x) \le
  \alpha}$.  Finally, a vector field $\map{f}{\real^n}{\real^n}$ is
said to be \emph{piecewise $\CC^2$} if it is continuous and there
exists (1) a finite collection of disjoint open sets $\DD_1, \dots,
\DD_m \subset \real^n$, referred to as \emph{patches}, whose closure
covers $\real^n$, that is, $\real^n = \cup_{i=1}^m \mathrm{cl}(\DD_i)$
and (2) a finite collection of $\CC^2$ functions
$\{\map{f_i}{\DD^e_i}{\real^n}\}_{i=1}^m$ where, for each $i \in
\until{m}$, $\DD^e_i$ is open with $\mathrm{cl}({\DD_i}) \subset
\DD^e_i$, such that $f_{|\mathrm{cl}(\DD_i)}$ and $f_i$ take the same
values over $\mathrm{cl}(\DD_i)$.
 
\subsection{Proximal calculus}\label{sec:proximal-calculus}

We present here a few notions on proximal calculus
following~\cite{FHC-YSL-RJS-PRW:98}. Given a closed set
$\EE \subset \real^n$ and a point $x \in \real^n \backslash \EE$, the
distance from $x$ to $\EE$ is,
\begin{equation}\label{eq:distance}
  d_{\EE}(x) = \min_{y \in \EE} \norm{x - y}.
\end{equation}
We let $\proj_{\EE} (x)$ denote the set of points in $\EE$ that are
closest to $x$, i.e.,
$\proj_{\EE}(x) = \setdef{y \in \EE}{\norm{x-y} = d_{\EE}(x)} \subset
\EE$.
For $y \in \proj_{\EE}(x)$, the vector $x - y$ is a \emph{proximal
  normal direction} to $\EE$ at $y$ and any nonnegative multiple
$\zeta = t(x-y)$, $t \ge 0$ is called a \emph{proximal normal}
($P$-normal) to $\EE$ at $y$.
The distance function $d_{\EE}$ might not be differentiable in general
(unless $\EE$ is convex), but is globally Lipschitz and
regular~\cite[p. 23]{FHC-YSL-RJS-PRW:98}.  For a locally Lipschitz
function $f:\real^n \to \real$, the \emph{generalized gradient}
$\partial f: \real^n \rrarrows \real^n$ is
\begin{equation*}
  \partial f(x) = \mathrm{co} \setdef{ \lim_{i \rightarrow \infty} 
    \gradient f(x_i)}{ x_i \rightarrow x, x_i \notin S \cup \Omega_f},
\end{equation*}
where $\mathrm{co}$ denotes convex hull, $S \subset \real^n$ is any
set of measure zero, and $\Omega_f$ is the set (of measure zero) of
points where $f$ is not differentiable. In the case of the square of
the distance function, one can
compute~\cite[p.~99]{FHC-YSL-RJS-PRW:98} the generalized gradient as,
\begin{align}\label{eq:gen-gradient-d}
  \partial d_{\EE}^2(x) = \mathrm{co}\setdef{2(x-y)}{y \in
  \proj_{\EE}(x)}.
\end{align}

\subsection{Saddle points}\label{sec:saddle-points}

Here, we provide basic definitions pertaining to the notion of saddle
points.  A point $(\xo,\zo) \in \real^n \times \real^m$ is a
\emph{local min-max saddle point} of a continuously differentiable
function $F:\real^n \times \real^m \rightarrow \real$ if there exist
open neighborhoods $\UU_{\xo} \subset \real^n$ of $\xo$ and $\UU_{\zo}
\subset \real^m$ of $\zo$ such that
\begin{align}\label{eq:saddleinequality}
  F(\xo, z) & \le F(\xo, \zo) \le F(x, \zo) ,
\end{align}
for all $ z \in \UU_{\zo}$ and $x \in \UU_{\xo}$.  The point $(\xo,
\zo)$ is a \emph{global min-max saddle point} of $F$ if $\UU_{\xo} =
\real^n$ and $\UU_{\zo}= \real^m$.  Min-max saddle points are a
particular case of the more general notion of saddle points. We focus
here on min-max saddle points motivated by problems in constrained
optimization and zero-sum games, whose solutions correspond to min-max
saddle points. With a slight abuse of terminology, throughout the
  paper we refer to the local min-max saddle points simply as saddle
  points. We denote by $\saddleset{F}$ the set of saddle points
of~$F$. From~\eqref{eq:saddleinequality}, for $(\xo,\zo) \in
\saddleset{F}$, the point $\xo \in \real^n$ (resp.  $\zo \in \real^m$)
is a local minimizer (resp. local maximizer) of the map $x \mapsto
F(x,\zo)$ (resp. $z \mapsto F(\xo,z)$).  Each saddle point is a
critical point of $F$, i.e., $\gradient_x F(\xo,\zo) =0$ and
$\gradient_z F(\xo,\zo) =0$.  Additionally, if $F$ is $\CC^2$, then
$\gradient_{xx} F(\xo,\zo) \preceq 0$ and $\gradient_{zz} F(\xo,\zo)
\succeq 0$.  Also, if $\gradient_{xx} F(\xo,\zo) \prec 0$ and
$\gradient_{zz} F(\xo,\zo) \succ 0$, then the inequalities
in~\eqref{eq:saddleinequality} are strict.

A function $F:\real^n \times \real^m \to \real$ is \emph{locally
  convex-concave} at a point $(\xt,\zt) \in \real^n \times \real^m$ if
there exists an open neighborhood $\UU$ of $(\xt, \zt)$ such that for
all $(\bar{x}, \bar{z}) \in \UU$, the functions
$x \mapsto F(x,\bar{z})$ and $z \mapsto F(\bar{x},z)$ are convex over
$\UU \cap (\real^n \times \{\bar{z}\})$ and concave over
$\UU \cap (\{\bar{x}\} \times \real^m)$, respectively. If in addition,
either $x \mapsto F(x, \zt)$ is strictly convex in an open
neighborhood of $\xt$, or $z \mapsto F(\xt, z)$ is strictly concave in
an open neighborhood of $\zt$, then $F$ is \emph{locally strictly
  convex-concave} at $(\xt,\zt)$. $F$ is locally (resp. locally
strictly) convex-concave on a set $\SS \subset\real^n \times \real^m$
if it is so at each point in $\SS$. $F$ is \emph{globally
  convex-concave} if in the local definition
$\UU = \real^n \times \real^m$. Finally, $F$ is \emph{globally
  strictly convex-concave} if it is globally convex-concave and for
any $(\bar{x},\bar{z}) \in \real^n \times \real^m$, either
$x \mapsto F(x,\bar{z})$ is strictly convex or
$z \mapsto F(\bar{x},z)$ is strictly concave. Note that this notion is
different than saying that $F$ is \emph{both} strictly convex and
strictly concave.

Next, we define strongly quasiconvex function following~\cite{MVJ:96}.
A function $\map{f}{\real^n}{\real}$ is \emph{strongly quasiconvex}
with parameter $s>0$ over a convex set $\DD \subset \real^n$ if for
all $x,y \in \DD$ and all $\lambda \in [0,1]$ we have,
\begin{align*}
  \max\{f(x),f(y)\} - f(\lambda x + (1-\lambda) y) \ge s \lambda
  (1-\lambda) \norm{x-y}^2.
\end{align*}
A function $f$ is \emph{strongly quasiconcave} with parameter $s>0$
over the set $\DD$ if $-f$ is strongly quasiconvex with parameter $s$
over $\DD$.  A function $F:\real^n \times \real^m \to \real$ is
\emph{locally jointly strongly quasiconvex-quasiconcave} at a point
$(\xt,\zt) \in \real^n \times \real^m$ if there exist $s>0$ and an
open neighborhood $\UU$ of $(\xt, \zt)$ such that for all
$(\bar{x}, \bar{z}) \in \UU$, the function $x \mapsto F(x,\bar{z})$ is
strongly quasiconvex with parameter $s$ over
$\UU \cap (\real^n \times \{\bar{z}\})$ and the function
$z \mapsto F(\bar{x},z)$ is strongly quasiconvex with parameter $s$
over $\UU \cap (\{\bar{x}\} \times \real^m)$. $F$ is locally jointly
strongly quasiconvex-quasiconcave on a set
$\SS \subset \real^n \times \real^m$ if it is so at each point in
$\SS$. $F$ is \emph{globally jointly strongly
  quasiconvex-quasiconcave} if in the local definition
$\UU = \real^n \times \real^m$.

\section{Problem statement}\label{sec:problem}

Here we formulate the problem of interest in the paper.  Given a
continuously differentiable function $F: \real^n \times \real^m \to
\real$, which we refer to as \emph{saddle function}, we consider its
saddle-point dynamics, i.e., gradient-descent in one argument and
gradient-ascent in the other,
\begin{subequations}\label{eq:saddledynamics}
  \begin{align}
    \dot x & = - \gradient_{x} F(x,z), 
    \\
    \dot z & = \gradient_{z} F(x,z).
  \end{align}
\end{subequations}
When convenient, we use the shorthand notation $\SD: \real^n \times
\real^m \rightarrow \real^n \times \real^m$ to refer to this dynamics.
Our aim is to provide conditions on $F$ under which the trajectories
of its saddle-point dynamics~\eqref{eq:saddledynamics} locally
asymptotically converge to its set of saddle points, and possibly to a
point in the set.  We are also interested in identifying conditions to
establish global asymptotic convergence. Throughout our study, we
assume that the set $\saddleset{F}$ is nonempty. This
  assumption is valid under mild conditions in the application areas
  that motivate our study: for the Lagrangian of the constrained
  optimization problem~\cite{SB-LV:04} and the value function for
  zero-sum games~\cite{TB-GJO:82}. Our forthcoming discussion is
divided in two threads, one for the case of convex-concave functions,
cf.  Section~\ref{sec:case1}, and one for the case of general
functions, cf.  Section~\ref{sec:case2}. 
In each case, we provide illustrative examples to show the
applicability of the results.

\section{Convergence analysis for convex-concave saddle
  functions}\label{sec:case1}

This section presents conditions for the asymptotic stability of
saddle points under the saddle-point
dynamics~\eqref{eq:saddledynamics} that rely on the
convexity-concavity properties of the saddle function.

\subsection{Stability under strict
convexity-concavity}\label{subsec:cc}

Our first result provides conditions that guarantee the local
asymptotic stability of the set of saddle points.

\begin{proposition}\longthmtitle{Local asymptotic stability of the set
    of saddle points via convexity-concavity}\label{pr:localsaddlesetconv1}
  For $\map{F}{\real^n \times \real^m}{\real}$ continuously
  differentiable and locally strictly convex-concave on
  $\saddleset{F}$, each isolated path connected component of
  $\saddleset{F}$ is locally asymptotically stable under the
  saddle-point dynamics $\SD$ and, moreover, the convergence of each
  trajectory is to a point.
\end{proposition}
\begin{proof}
  Let $\SS$ be an isolated path connected component of $\saddleset{F}$
  and take $(\xo,\zo) \in \SS$. Without loss of generality, we
  consider the case when $x \mapsto F(x, \zo)$ is locally strictly
  convex (the proof for the case when $z \mapsto F(\xo, z)$ is locally
  strictly concave is analogous).  Consider the function
  $V:\real^n \times \real^m \to \realnonnegative$,
  \begin{equation}\label{eq:euclyapunov}
    V(x,z) = \frac{1}{2} \Big( \norm{x-\xo}^2 + \norm{z-\zo}^2 \Big),
  \end{equation}
  which we note is radially unbounded (and hence has bounded sublevel
  sets).  We refer to $V$ as a LaSalle function because
    locally, as we show next, its Lie derivative is negative, but not
    strictly negative. Let $\UU$ be the neighborhood of $(\xo,\zo)$
  where local convexity-concavity holds. The Lie derivative of~$V$
  along the dynamics~\eqref{eq:saddledynamics} at $(x,z) \in \UU $ can
  be written as,
  \begin{align}
    \Lie_{\SD} V (x,z) 
    & = -(x - \xo)^\top \gradient_{x} F(x,z) + (z - \zo)^\top
      \gradient_{z} F(x,z) \label{eq:liesaddledynamics} 
    \\
    &  \le F(\xo, z) - F(x,z) + F(x,z) - F(x, \zo) \notag
    \\
    & = F(\xo, z) - F(\xo, \zo) + F(\xo, \zo) - F(x, \zo) \le 0,
      \notag
  \end{align}
  where the first inequality follows from the first-order condition
  for convexity and concavity, and the last inequality follows from
  the definition of saddle point. As a consequence, for $\alpha>0$
  small enough such that $\levelset{V}{\alpha} \subset \UU$, we
  conclude that $\levelset{V}{\alpha}$ is positively invariant
  under~$\SD$.  The application of the LaSalle Invariance
  Principle~\cite[Theorem 4.4]{HKK:02} yields that any trajectory
  starting from a point in $\levelset{V}{\alpha}$ converges to the
  largest invariant set $M$ contained in $\setdef{(x, z) \in
    \levelset{V}{\alpha}}{\Lie_{\SD} V(x,z) = 0}$.  Let $(x,z) \in
  M$. From~\eqref{eq:liesaddledynamics}, $\Lie_{\SD} V(x,z) =0$
  implies that $F(\xo, z) = F(\xo, \zo)= F(x, \zo)$. In turn, the
  local strict convexity of $x \mapsto F(x, \zo)$ implies that $x =
  \xo$. Since $M$ is positively invariant, the trajectory $t \mapsto
  (x(t),z(t))$ of $\SD$ starting at $(x,z)$ is contained in~$M$. This
  implies that along the trajectory, for all $t \ge 0$, (a) $x(t) =
  \xo$ i.e., $\dot x(t) = \gradient_x F(x(t),z(t)) = 0$, and (b)
  $F(\xo,z(t)) = F(\xo,\zo)$. The later implies
  \begin{align*}
    0 = \Lie_{\SD} F(\xo, z(t)) = \SD (\xo,z(t)) \cdot (0,
      \nabla_z F(\xo, z(t))) = \norm{\gradient_z F(x(t) , z(t))}^2 ,
  \end{align*}
  for all $t \ge 0$. Thus, we get $\gradient_x F(x,z) = 0$ and
  $\gradient_z F(x,z) = 0$. Further, since $(x,z) \in \UU$, local
  convexity-concavity holds over $\UU$, and $\SS$ is an isolated
  component, we obtain $(x,z) \in \SS$, which shows $M \subset \SS$.
  Since $(\xo,\zo)$ is arbitrary, the asymptotic convergence property
  holds in a neighborhood of $\SS$. The pointwise convergence follows
  from the application of Lemma~\ref{le:convtopoint}.
\end{proof}

The result above shows that each saddle point is stable and that each
path connected component of $\saddleset{F}$ is asymptotically
stable. Note that each saddle point might not be asymptotically
stable. However, if a component consists of a single point, then that
point is asymptotically stable.  Interestingly, a close look at the
proof of Proposition~\ref{pr:localsaddlesetconv1} reveals that, if the
assumptions hold globally, then the asymptotic stability of the set of
saddle points is also global, as stated next.

\begin{corollary}\longthmtitle{Global asymptotic stability of
    the set of saddle points via
    convexity-concavity}\label{cr:globalsaddlesetconv1} 
  For $\map{F}{\real^n \times \real^m}{\real}$ continuously
  differentiable and globally strictly convex-concave, $\saddleset{F}$
  is globally asymptotically stable under the saddle-point dynamics
  $\SD$ and the convergence of trajectories is to a point.
\end{corollary}

\begin{remark}\longthmtitle{Relationship with results on primal-dual
    dynamics: I}\label{re:strictlyconvex}
  {\rm Corollary~\ref{cr:globalsaddlesetconv1} is an extension to more
    general functions and less stringent assumptions of the results
    stated for Lagrangian functions of constrained convex (or concave)
    optimization problems in~\cite{JW-NE:11,KA-LH-HU:58,DF-FP:10} and
    cost functions of differential games in~\cite{LJR-SAB-SSS:13}.
    In~\cite{KA-LH-HU:58,DF-FP:10}, for a concave optimization, the
    matrix $\gradient_{xx} F$ is assumed to be negative definite at
    every saddle point and in~\cite{JW-NE:11} the set $\saddleset{F}$
    is assumed to be a singleton.  The work~\cite{LJR-SAB-SSS:13}
    assumes a sufficient condition on the cost functions to guarantee
    convergence that in the current setup is equivalent to having
    $\gradient_{xx} F$ and $\gradient_{zz} F$ positive and negative
    definite, respectively.}  \oprocend
\end{remark}

\subsection{Stability under convexity-linearity or
linearity-concavity}\label{subsec:linearity}

Here we study the asymptotic convergence properties of the
saddle-point dynamics when the convexity-concavity of the saddle
function is not strict but, instead, the function depends linearly on
its second argument. The analysis follows analogously for saddle
functions that are linear in the first argument and concave in the
other.  The consideration of this class of functions is motivated by
equality constrained optimization problems.

\begin{proposition}\longthmtitle{Local asymptotic stability of the set
    of saddle points via
    convexity-linearity}\label{pr:localsaddlesetconv2}
  For a continuously differentiable function $\map{F}{\real^n \times
    \real^m}{\real}$, if
  \begin{enumerate}
  \item $F$ is locally convex-concave on $\saddleset{F}$ and linear in
    $z$,
    \label{as:2ls2}
  \item for each $(\xo,\zo) \in \saddleset{F}$, there exists a
    neighborhood $\UU_{\xo} \subset \real^n$ of $\xo$ where, if
    $F(x,\zo) = F(\xo,\zo)$ with $x \in \UU_{\xo}$, then
    $(x,\zo) \in \saddleset{F}$,
    \label{as:2ls3}
  \end{enumerate}
  then each isolated path connected component of $\saddleset{F}$ is
  locally asymptotically stable under the saddle-point dynamics $\SD$
  and, moreover, the convergence of trajectories is to a point.
\end{proposition}
\begin{proof}
  Given an isolated path connected component $\SS$ of $\saddleset{F}$,
  Lemma~\ref{le:Fconstant} implies that $F_{|\SS}$ is constant.  Our
  proof proceeds along similar lines as those of
  Proposition~\ref{pr:localsaddlesetconv1}.  With the same notation,
  given $(\xo,\zo) \in \SS$, the arguments follow verbatim until the
  identification of the largest invariant set $M$ contained in
  $\setdef{(x,z) \in \levelset{V}{\alpha}}{\Lie_{\SD} V(x,z) =0}$.
  Let $(x,z) \in M$.  From~\eqref{eq:liesaddledynamics}, $\Lie_{\SD}
  V(x,z) = 0$ implies $F(\xo,z) = F(\xo,\zo) = F(x,\zo)$.  By
  assumption~\ref{as:2ls3}, this means $(x,\zo) \in \SS$, and by
  assumption~\ref{as:2ls2}, the linearity property gives $\gradient_z
  F(x,z) = \gradient_z F(x,\zo) = 0$. Therefore $\gradient_z F_{|M} =
  0$. For $(x,z) \in M$, the trajectory $t \mapsto (x(t),z(t))$ of
  $\SD$ starting at $(x,z)$ is contained in~$M$. Consequently, $z(t) =
  z$ for all $t \in [0,\infty)$ and $\dot x(t) = -\gradient_x
  F(x(t),z)$ corresponds to the gradient dynamics of the (locally)
  convex function $y \mapsto F(y,z)$.  Therefore, $x(t)$ converges to
  a minimizer $x'$ of this function, i.e., $\nabla_x F(x',z) =
  0$. Since $\gradient_z F_{|M} = 0$, the continuity of $\gradient_z
  F$ implies that $\nabla_z F(x',z) = 0$, and hence $(x',z) \in
  \SS$. By continuity of $F$, it follows that $F(x(t),z) \to F(x',z) =
  F(x_*,z_*)$, where for the equality we use the fact that $F_{|\SS}$
  is constant. On the other hand, note that $0 = \Lie_{\SD} V(x(t),z)
  = -(x(t)-\xo)^\top \gradient_x F(x(t),z) \le F(\xo,z) - F(x(t),z)$
  implies
  \begin{equation*}
    F(x(t),z) \le F(\xo,z) = F(\xo,\zo) ,
  \end{equation*}
  for all $ t \in [0,\infty)$. Therefore, the monotonically
  nonincreasing sequence $\{F(x(t),z)\}$ converges to $ F(\xo,\zo)$,
  which is also an upper bound on the whole sequence. This can only be
  possible if $F(x(t),z) = F(\xo,\zo)$ for all $t \in [0,\infty)$.
  This further implies $\gradient_x F(x(t),z) = 0$ for all
  $t \in [0,\infty)$, and hence, $(x,z) \in \SS$. Consequently,
  $M \subset \SS$.  Since $(\xo,\zo)$ has been chosen arbitrarily, the
  convergence property holds in a neighborhood of $\SS$.  The
  pointwise convergence follows now from the application of
  Lemma~\ref{le:convtopoint}.
\end{proof}

The assumption (ii) in the above result is a generalization of
  the local strict convexity condition for the function
  $F(\cdot,\zo)$. That is, (ii) allows other points in the
  neighborhood of $\xo$ to have the same value of the function
  $F(\cdot,\zo)$ as that at $\xo$, as long as they are saddle points
  (whereas, under local strict convexity, $\xo$ is the local unique
  minimizer of $F(\cdot,\zo)$). The next result extends the
conclusions of Proposition~\ref{pr:localsaddlesetconv2} globally when
the assumptions hold globally.

\begin{corollary}\longthmtitle{Global asymptotic stability of the set
    of saddle points via
    convexity-linearity}\label{cr:globalsaddlesetconv2}
  For a $\CC^1$ function $\map{F}{\real^n \times \real^m}{\real}$, if
  \begin{enumerate}
  \item $F$ is globally convex-concave and linear in $z$,
    \label{as:2gs2}
  \item for each $(\xo,\zo) \in \saddleset{F}$, if
    $F(x,\zo) = F(\xo,\zo)$, then $(x,\zo) \in \saddleset{F}$,
    \label{as:2gs3}
  \end{enumerate}
  then $\saddleset{F}$ is globally asymptotically stable under the
  saddle-point dynamics $\SD$ and, moreover, convergence of
  trajectories is to a point.
\end{corollary}

\begin{example}\longthmtitle{Saddle-point dynamics for convex
    optimization}\label{ex:convexconcave}
  {\rm Consider the following convex optimization problem
    on~$\real^3$,
    \begin{subequations}\label{eq:convexconopt}
      \begin{align}
        \mathrm{minimize} & \quad  (x_1 + x_2 + x_3)^2,
        \\
        \text{subject to} & \quad x_1 = x_2.
      \end{align}
    \end{subequations}
    The set of solutions of this optimization is
    $\setdef{x \in \real^3}{2x_1 + x_3 = 0, x_2 = x_1}$, with
    Lagrangian
    \begin{align}\label{eq:augmented}
      L(x,z) = (x_1+x_2+x_3)^2 + z(x_1 - x_2) ,
    \end{align}
    where $z \in \real$ is the Lagrange multiplier. The set of saddle
    points of $L$ (which correspond to the set of primal-dual
    solutions to~\eqref{eq:convexconopt}) are $\saddleset{L} =
    \setdef{(x,z) \in \real^3 \times \real}{2x_1 + x_3 = 0, x_1 = x_2,
      \text{ and } z=0}$.  However, $L$ is not strictly convex-concave
    and hence, it does not satisfy the hypotheses of
    Corollary~\ref{cr:globalsaddlesetconv1}.  While $L$ is globally
    convex-concave and linear in $z$, it does not satisfy assumption
    (ii) of Corollary~\ref{cr:globalsaddlesetconv2}. Therefore, to
    identify a dynamics that renders $\saddleset{L}$ asymptotically
    stable, we form the augmented Lagrangian
    \begin{equation}\label{eq:Ltilde}
      \tilde{L}(x,z) = L(x,z) + (x_1 - x_2)^2 ,
    \end{equation} 
    that has the same set of saddle points as $L$. Note that
    $\tilde{L}$ is not strictly convex-concave but it is globally
    convex-concave (this can be seen by computing its Hessian) and is
    linear in $z$. Moreover, given any $(\xo,\zo) \in \saddleset{L}$,
    we have $\tilde{L}(\xo,\zo) = 0$, and if $\tilde{L}(x,\zo) =
    \tilde{L}(\xo,\zo) = 0$, then $(x,\zo) \in \saddleset{L}$.  By
    Corollary~\ref{cr:globalsaddlesetconv2}, the trajectories of the
    saddle-point dynamics of $\tilde{L}$ converge to a point in $\SS$
    and hence, solve the optimization problem~\eqref{eq:convexconopt}.
    Figure~\ref{fig:convex-opt-example} illustrates this fact.
    Note that the point of convergence depends on the
    initial condition. }
  \oprocend
  \begin{figure}[htb!]
    \centering
    \subfigure[$(x,z)$]{\includegraphics[width=.44\linewidth]{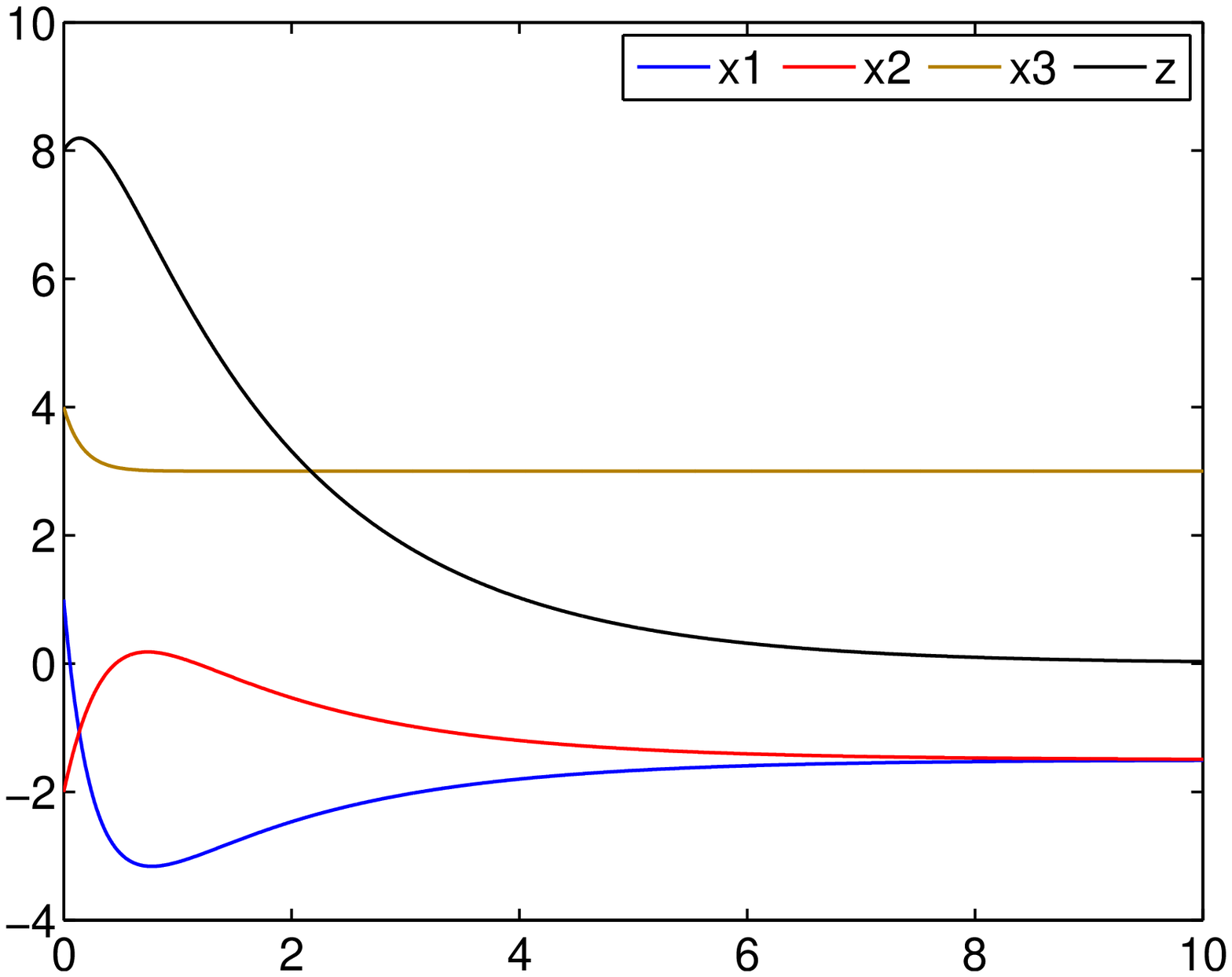}}
    \quad
    \subfigure[$(x_1+x_2+x_3)^2$]{\includegraphics[width=.44\linewidth]{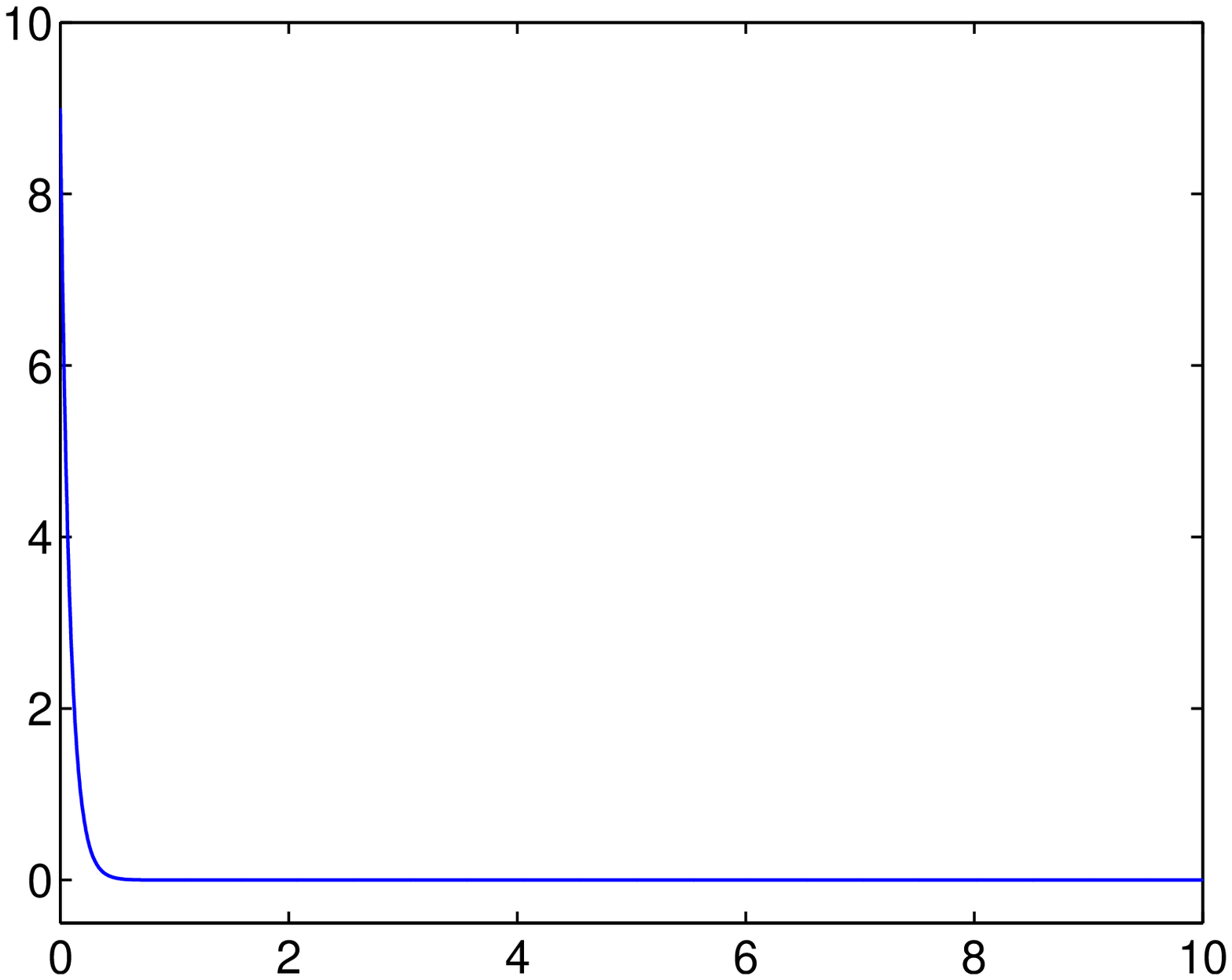}}
    \caption{(a) Trajectory of the saddle-point dynamics of the
      augmented Lagrangian $\tilde{L}$ in~\eqref{eq:Ltilde} for the
      optimization problem~\eqref{eq:convexconopt}. The initial
      condition is $(x,z) = (1,-2,4,8)$. The trajectory converges to
      $(-1.5,-1.5,3,0) \in \saddleset{L}$. (b) Evolution of the
      objective function of the optimization~\eqref{eq:convexconopt}
      along the trajectory.  The value converges to the
      minimum,~$0$.}\label{fig:convex-opt-example}
    \vspace*{-1ex}
  \end{figure}
\end{example}

\begin{remark}\longthmtitle{Relationship with results on primal-dual
    dynamics: II}\label{re:nonstrict}
  {\rm The work~\cite[Section 4]{DF-FP:10} considers concave
    optimization problems under inequality constraints where the
    objective function is not strictly concave but analyzes the
    convergence properties of a different dynamics. Specifically, the
    paper studies a discontinuous dynamics based on the saddle-point
    information of an augmented Lagrangian combined with a projection
    operator that restricts the dual variables to the nonnegative
    orthant.  We have verified that, for the formulation of the
    concave optimization problem in~\cite{DF-FP:10} but with equality
    constraints, the augmented Lagrangian satisfies the hypotheses of
    Corollary~\ref{cr:globalsaddlesetconv2}, implying that the
    dynamics~$\SD$ renders the primal-dual optima of the problem
    asymptotically stable. } \oprocend
\end{remark}

\subsection{Stability under strong
  quasiconvexity-quasiconcavity}\label{case-quasi}

Motivated by the aim of further relaxing the conditions for asymptotic
convergence, we conclude this section by weakening the
convexity-concavity requirement on the saddle function. The next
result shows that strong quasiconvexity-quasiconcavity is sufficient
to ensure convergence of the saddle-point dynamics.

\begin{proposition}\longthmtitle{Local asymptotic stability of the set
    of saddle points via strong
    quasiconvexity-quasiconcavity}\label{pr:quasi}
  Let $\map{F}{\real^n \times \real^m}{\real}$ be $\CC^2$ and the map
  $(x,z) \mapsto \gradient_{xz} F(x,z)$ be locally Lipschitz.  Assume
  that $F$ is locally jointly strongly quasiconvex-quasiconcave on
  $\saddleset{F}$. Then, each isolated path connected component of
  $\saddleset{F}$ is locally asymptotically stable under the
  saddle-point dynamics $\SD$ and, moreover, the convergence of
  trajectories is to a point.  Further, if $F$ is globally jointly
  strongly quasiconvex-quasiconcave and $\gradient_{xz} F$ is constant
  over $\real^n \times \real^m$, then $\saddleset{F}$ is globally
  asymptotically stable under $\SD$ and the convergence of trajectories
  is to a point.
\end{proposition}
\begin{proof}
  Let $(\xo,\zo) \in \SS$, where $\SS$ is an isolated path connected
  component of $\saddleset{F}$, and consider the function
  $\map{V}{\real^n \times \real^m}{\realnonnegative}$ defined
  in~\eqref{eq:euclyapunov}.
  Let $\UU$ be the neighborhood of $(\xo,\zo)$ where the local joint strong
  quasiconvexity-quasiconcavity holds.   The Lie derivative of $V$ along 
  the saddle-point dynamics at $(x,z) \in \UU$ can be written as,
  \begin{align}
    \Lie_{\SD} V (x,z) & = -(x - \xo)^\top \gradient_{x} F(x,z) + (z -
    \zo)^\top \gradient_{z} F(x,z), \notag
    \\
    & = -(x - \xo)^\top \gradient_{x} F(x,\zo) + (z - \zo)^\top
    \gradient_{z} F(\xo,z) + M_1 + M_2, \label{eq:sq-lie}
  \end{align}
  where 
  \begin{align*}
    M_1 &= - (x-\xo)^\top (\gradient_x F(x,z) - \gradient_x
    F(x,\zo)),  \\
    M_2 &= (z-\zo)^\top (\gradient_z F(x,z) - \gradient_z F(\xo,z)).
  \end{align*}
  Writing
  \begin{align*}
    \gradient_x F(x,z) - \gradient_x F(x,\zo) & = \int_0^1
    \gradient_{zx} F(x,\zo+t(z-\zo)) (z-\zo) dt,
    \\
    \gradient_z F(x,z) - \gradient_z F(\xo,z) & = \int_0^1
    \gradient_{xz} F(\xo + t(x-\xo),z) (x-\xo) dt,
  \end{align*}
  we get
  \begin{align}
    M_1 + M_2 & =(z-\zo)^\top \Big( \int_0^1 \big( \gradient_{xz}
    F(\xo+t(x-\xo), z) \notag
    \\
    & \qquad \qquad \qquad \qquad \qquad - \gradient_{xz}
    F(x,\zo+t(z-\zo))\big) dt \Big) (x-\xo) \notag
    \\
    & \le \norm{z-\zo}(L \norm{x-\xo} + L \norm{z-\zo})
    \norm{x-\xo}, \label{eq:mineq}
  \end{align}
  where in the inequality, we have used the fact that
  $\gradient_{xz}F$ is locally Lipschitz with some constant $L >0$.
  From the first-order property of a strong quasiconvex function,
  cf. Lemma~\ref{le:first-quasi}, there exist constants $s_1,s_2>0$
  such that
  \begin{subequations}\label{eq:sq-bound}
    \begin{align}
      -(x - \xo)^\top \gradient_{x} F(x,\zo) \le -s_1 \norm{x-\xo}^2,
      \\
      (z - \zo)^\top \gradient_{z} F(\xo,z) \le -s_2 \norm{z-\zo}^2,
    \end{align}
  \end{subequations}
  for all $(x,z) \in \UU$.  Substituting~\eqref{eq:mineq}
  and~\eqref{eq:sq-bound} into the expression for the Lie
  derivative~\eqref{eq:sq-lie}, we obtain
  \begin{align*}
    \Lie_{\SD} V (x,z) \le -s_1 \norm{x-\xo}^2 -s_2 \norm{z-\zo}^2 + L
    \norm{x-\xo}^2 \norm{z-\zo} + L \norm{x-\xo} \norm{z-\zo}^2.
  \end{align*}
  To conclude the proof, note that if $\norm{z-\zo} < \frac{s_1}{L}$
  and $\norm{x-\xo} < \frac{s_2}{L}$, then $\Lie_{\SD} V (x,z) < 0$,
  which implies local asymptotic stability. The pointwise convergence 
  follows from Lemma~\ref{le:convtopoint}. The global asymptotic
  stability can be reasoned using similar arguments as above using the
  fact that here $M_1 + M_2 = 0$ because $\gradient_{xz} F$ is constant.
\end{proof}

In the following, we present an example where the above result
  is employed to explain local asymptotic convergence. In this case,
  none of the results from Section~\ref{subsec:cc}
  and~\ref{subsec:linearity} apply, thereby justifying the importance
  of the above result.
  \begin{example}\longthmtitle{Convergence for locally jointly
      strongly quasiconvex-quasiconcave function}\label{ex:qqf}
    {\rm Consider $\map{F}{\real \times \real}{\real}$ given by,
  \begin{equation}\label{eq:quasiF}
    F(x,z) = (2-e^{-x^2})(1+e^{-z^2}).
  \end{equation}
  Note that $F$ is $\CC^2$ and $\gradient_{xz} F(x,z) =
  -4xze^{-x^2}e^{-z^2}$ is locally Lipschitz. To see this, note that
  the function $x \mapsto x e^{-x^2}$ is bounded and is locally
  Lipschitz (as its derivative is bounded). Further, the product of
  two bounded and locally Lipschitz functions is locally
  Lipschitz~\cite[Theorem 4.6.3]{HHS:03} and so, $(x,z) \mapsto
  \gradient_{xz} F(x,z)$ is locally Lipschitz. The set of saddle
  points of $F$ is $\saddleset{F} = \{0\}$. Next, we show that $x
  \mapsto f(x) = c_1 - c_2 e^{-x^2}$, $c_2 > 0$, is locally strongly
  quasiconvex at $0$. Fix $\delta > 0$ and let $x, y \in B_\delta(0)$
  such that $f(y) \le f(x)$. Then, $\abs{y} \le \abs{x}$ and
  \begin{align*}
    &\max\{f(x),f(y)\} - f(\lambda x + (1-\lambda) y) - s \lambda (1-
    \lambda ) (x-y)^2
    \\
    & \qquad \qquad \qquad \qquad = c_2(-e^{-x^2} + e^{-(\lambda x +
      (1-\lambda)y)^2}) - s \lambda (1-\lambda) (x-y)^2
    \end{align*}
    % \\
    \begin{align*}
    & \qquad \qquad \qquad \qquad= c_2 e^{-x^2}(-1 + e^{x^2 - (\lambda
      x + (1-\lambda)y)^2}) - s \lambda (1-\lambda) (x-y)^2
    \\
    & \qquad \qquad \qquad \qquad\ge c_2 e^{-x^2}(x^2 - (\lambda x +
    (1 - \lambda)y)^2) - s \lambda (1-\lambda) (x-y)^2
    \\
    & \qquad \qquad \qquad \qquad= (1-\lambda)(x-y)\Bigl(c_2
    e^{-x^2}(x+y) + \lambda (x-y)(c_2 e^{-x^2} - s)\Bigr) \ge 0,
  \end{align*}
  for $s \le c_2 e^{-\delta^2}$, given the fact that $\abs{y} \le
  \abs{x}$. Therefore, $f$ is locally strongly quasiconvex and so $-f$
  is locally strongly quasiconcave. Using these facts, we deduce that
  $F$ is locally jointly strongly quasiconvex-quasiconcave. Thus, the
  hypotheses of Proposition~\ref{pr:quasi} are met, implying local
  asymptotic stability of $\saddleset{F}$ under the saddle-point
  dynamics. Figure~\ref{fig:quasi-example} illustrates this fact in
  simulation.  Note that $F$ does not satisfy the conditions outlined
  in results of Section~\ref{subsec:cc} and~\ref{subsec:linearity}.}
  \oprocend

\begin{figure}[htb!]
    \centering
    \subfigure[$(x,z)$]{\includegraphics[width=.44\linewidth]{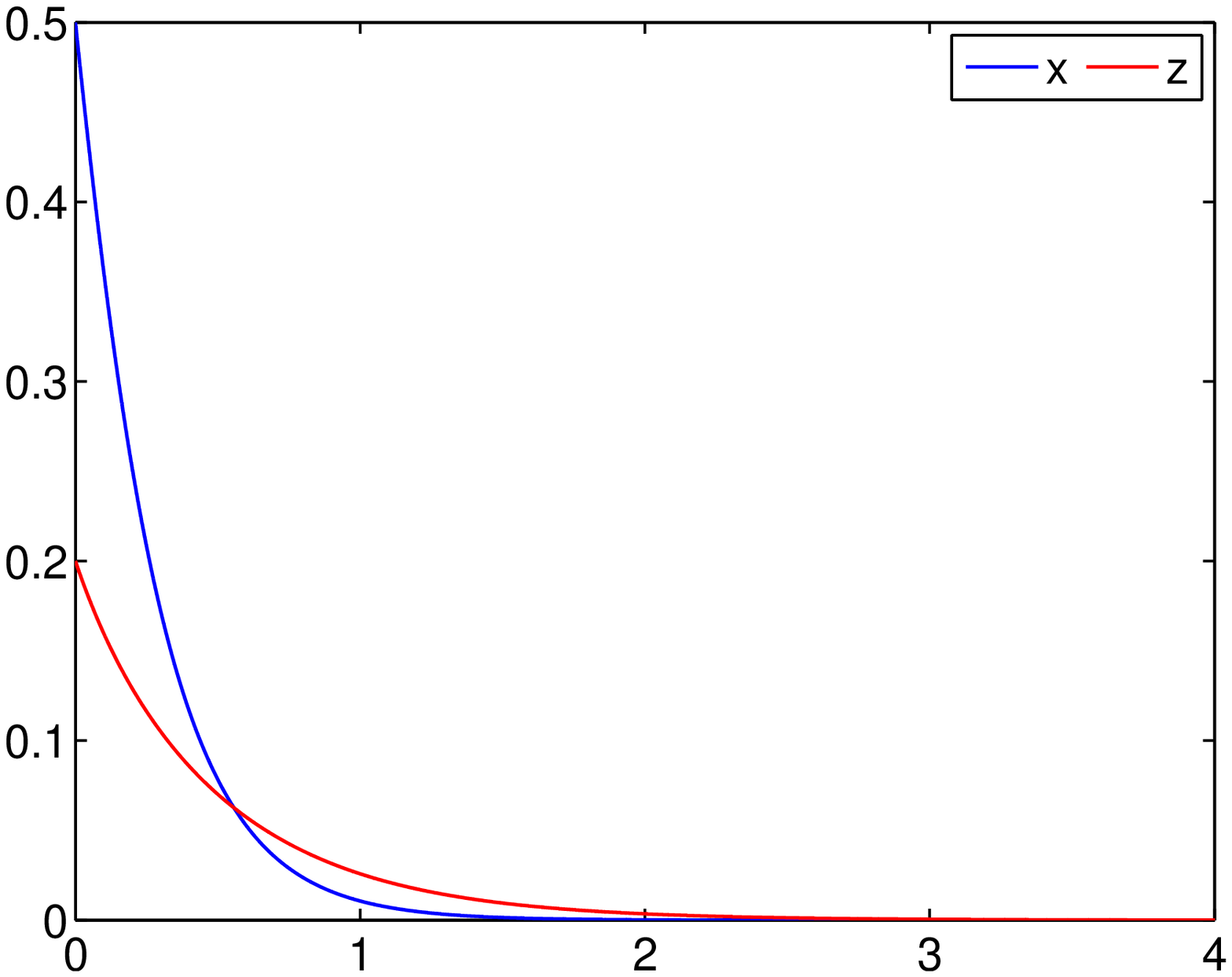}}
    \quad
    \subfigure[$(x^2+
    z^2)/2$]{\includegraphics[width=.44\linewidth]{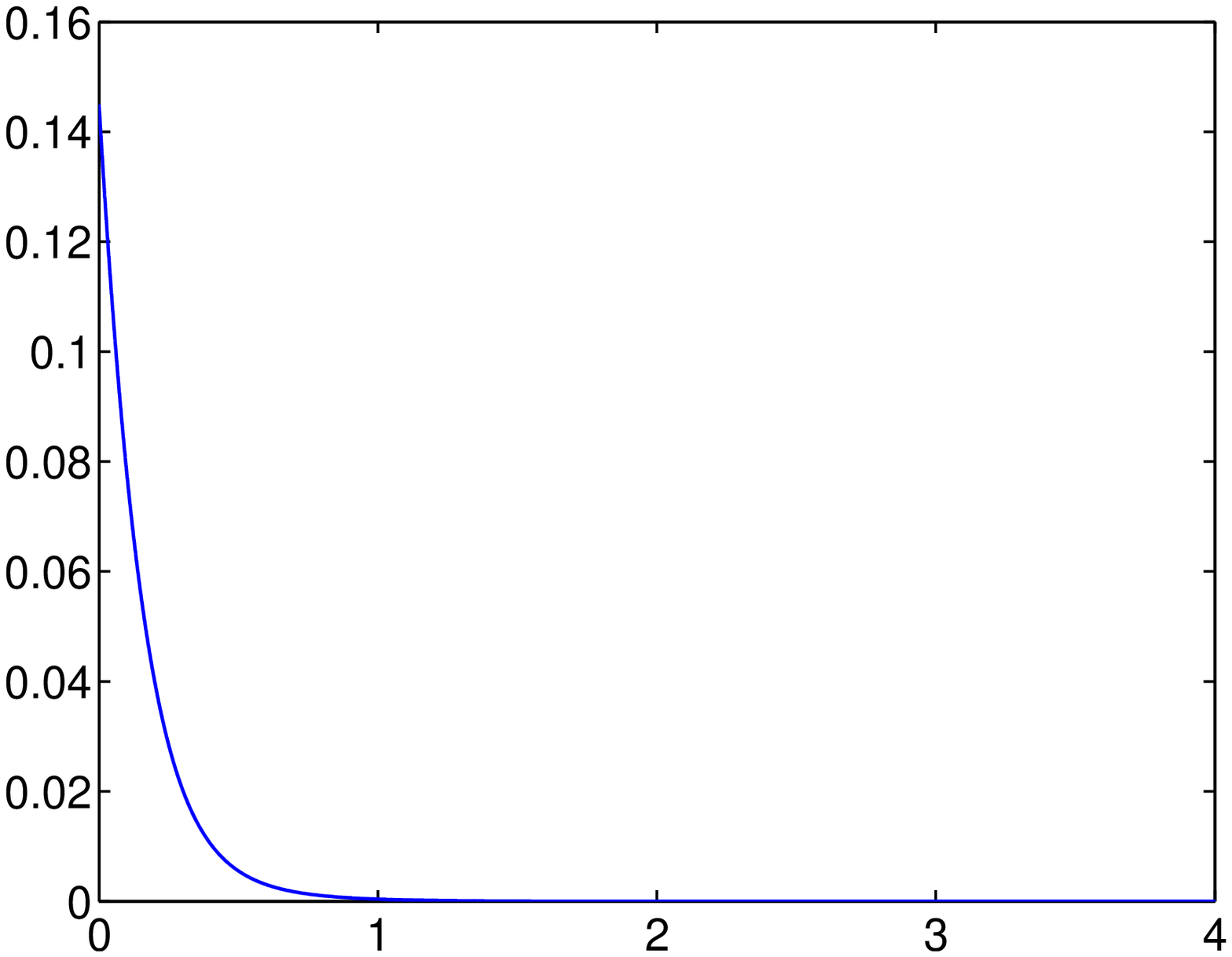}}
    \caption{(a) Trajectory of the saddle-point dynamics for $F$ given
      in~\eqref{eq:quasiF}. The initial
      condition is $(x,z) = (0.5,0.2)$. The trajectory converges to the
      saddle point $(0,0)$. (b) Evolution of the function $V$
      along the trajectory.}\label{fig:quasi-example}
    \vspace*{-1ex}
  \end{figure}
\end{example}

\section{Convergence analysis for general saddle functions}\label{sec:case2}

We study here the convergence properties of the saddle-point dynamics
associated to functions that are not convex-concave.  Our first
  result explores conditions for local asymptotic stability based on
  the linearization of the dynamics and properties of the
  eigenstructure of the Jacobian matrices. In particular, we assume
  that $\SD$ is piecewise $\CC^2$ and that the set of limit points of
  the Jacobian of $\SD$ at any saddle point have a common kernel and
  negative real parts for the nonzero eigenvalues. The proof is a
  direct consequence of Proposition~\ref{pr:convtoman-cl}.
\begin{proposition}\longthmtitle{Local asymptotic stability of
    manifold of saddle points via linearization -- piecewise $\CC^{3}$
    saddle function}\label{pr:localsaddlesetconv3}
  Given $\map{F}{\real^n \times \real^m}{\real}$, let
  $\SS \subset \saddleset{F}$ be a $p$-dimensional submanifold of
  saddle points. Assume that $F$ is $\CC^1$ with locally Lipschitz gradient
  on a neighborhood of $\SS$ and that the vector field $\SD$ is
  piecewise $\CC^2$. 
  Assume that at each $(\xo,\zo) \in
  \SS$, the set of matrices $\AA_* \subset \real^{n+m \times n+m}$
  defined as 
  \begin{align*}
    \AA_* = \setdef{\lim_{k \to
        \infty} D\SD(x_k,z_k)} 
    {(x_k,z_k) \to (x,z), (x_k,z_k) \in \real^{n+m} \setminus \Omega_{\SD}},
  \end{align*}
  where $\Omega_{\SD}$ is the set of points where $\SD$ is not
  differentiable, satisfies the following:
  \begin{enumerate}
  \item there exists an
    orthogonal matrix $Q \in \real^{n+m \times n+m}$ such that
    \begin{equation}\label{eq:q-transform}
      Q^\top A Q =
      \begin{bmatrix} 0 & 0 
        \\ 
        0 & \tilde{A} 
      \end{bmatrix},
    \end{equation}
    for all $A \in \AA_*$, where $\tilde{A} \in \real^{n+m-p \times
      n+m-p}$,
  \item the nonzero eigenvalues of the matrices in $\AA_*$ have
    negative real parts,
  \item there exists a positive definite matrix
    $P \in \real^{n+m-p \times n+m-p}$ such that
    \begin{align*}
      \tilde{A}^\top P + P \tilde{A} \prec
      0, 
    \end{align*}
    for all $\tilde{A}$ obtained by applying
    transformation~\eqref{eq:q-transform} on each $A \in \AA_*$.
  \end{enumerate}
  Then, $\SS$ is locally asymptotically stable
  under~\eqref{eq:fsystem-cl} and the trajectories converge to a point
  in $\SS$.
\end{proposition}

When $F$ is sufficiently smooth, we can refine the
above result as follows.

\begin{corollary}\longthmtitle{Local asymptotic stability of manifold
    of saddle points via linearization -- $\CC^{3}$ saddle
    function}\label{cr:localsaddlesetconv31}
  Given $\map{F}{\real^n \times \real^m}{\real}$, let
  $\SS \subset \saddleset{F}$ be a $p$-dimensional manifold of saddle
  points. Assume $F$ is $\CC^3$ on a neighborhood of $\SS$ and that
  the Jacobian of $\SD$ at each point in $\SS$ has no eigenvalues in
  the imaginary axis other than $0$, which is semisimple with
  multiplicity $p$. Then, $\SS$ is locally asymptotically stable under
  the saddle-point dynamics $\SD$ and the trajectories converge to a
  point.
\end{corollary}
\begin{proof}
  Since $F$ is $\CC^3$, the map $\SD$ is $\CC^2$ and so, the limit
  point of Jacobian matrices at a saddle point $(\xo,\zo) \in \SS$ is
  the Jacobian at that point itself, that is,
  \begin{align*}
    D \SD =
    \begin{bmatrix}
      -\gradient_{xx} F & -\gradient_{xz} F
      \\
      \gradient_{zx} F & \gradient_{zz} F
    \end{bmatrix}_{(\xo,\zo)}.
  \end{align*}
  From the definition of saddle point, we have $\gradient_{xx}
  F(\xo,\zo) \succeq 0$ and $\gradient_{zz} F(\xo,\zo) \preceq 0$.  In
  turn, we obtain $D \SD + D \SD^\top \preceq 0$, and since
  $\realpart(\lambda_i(D \SD)) \le \lambda_{\max} (\frac{1}{2} (D \SD
  + D \SD^\top))$~\cite[Fact 5.10.28]{DSB:05}, we deduce that
  $\realpart(\lambda_i (D \SD)) \le 0$. The statement now follows from
  Proposition~\ref{pr:localsaddlesetconv3} using the fact that the
  properties of the eigenvalues of $D\SD$ shown here imply existence
  of an orthonormal transformation leading to a form of $D\SD$ that
  satisfies assumptions {\it (i)-(iii)} of
  Proposition~\ref{pr:localsaddlesetconv3}.
\end{proof}

Next, we provide a sufficient condition under which the Jacobian of
$\SD$ for a saddle function $F$ that is linear in its second argument
satisfies the hypothesis of Corollary~\ref{cr:localsaddlesetconv31}
regarding the lack of eigenvalues on the imaginary axis other than
$0$.

\begin{lemma}\longthmtitle{Sufficient condition for absence of
    imaginary eigenvalues of the Jacobian of
    $\SD$}\label{le:eigenvalue}
  Let $\map{F}{\real^n \times \real^m}{\real}$ be $\CC^2$ and linear
  in the second argument. Then, the Jacobian of $\SD$ at any saddle
  point $(\xo,\zo)$ of $F$ has no eigenvalues on the imaginary axis
  except for $0$ if $\rge(\gradient_{zx} F(\xo,\zo)) \cap
  \nll(\gradient_{xx} F(\xo,\zo)) = \{0\}$.
\end{lemma}
\begin{proof}
  The Jacobian of $\SD$ at a saddle point $(\xo,\zo)$ for a saddle
  function $F$ that is linear in $z$ is given as
  \begin{align*}
    D \SD = \begin{bmatrix} A & B
      \\
      -B^\top & 0 \end{bmatrix},
  \end{align*}
  where $A = -\gradient_{xx} F(\xo,\zo)$ and $B = -\gradient_{zx}
  F(\xo,\zo)$. We reason by contradiction. Let $i\lambda$, $\lambda
  \not = 0$ be an imaginary eigenvalue of $D\SD$ with the
  corresponding eigenvector $a+ib$. Let $a = (a_1;a_2)$ and $b =
  (b_1;b_2)$ where $a_1, b_1 \in \real^n$ and $a_2, b_2 \in
  \real^m$. Then the real and imaginary parts of the condition $D \SD
  (a+ib) = (i \lambda )(a+ib)$ yield
  \begin{align}
    A a_1 + B a_2 = - \lambda b_1, & \qquad -B^\top a_1 = - \lambda b_2,
    \label{eq:eig-cond1}
    \\
    A b_1 + B b_2 = \lambda a_1, & \qquad -B^\top b_1 = \lambda
    a_2. \label{eq:eig-cond2}
  \end{align}
  Pre-multiplying the first equation of~\eqref{eq:eig-cond1} with
  $a_1^\top$ gives $a_1^\top A a_1 + a_1^\top B a_2 = - \lambda
  a_1^\top b_1$. Using the second equation of~\eqref{eq:eig-cond1}, we
  get $a_1^\top A a_1 = -\lambda (a_1^\top b_1 + a_2^\top b_2)$. A
  similar procedure for the set of equations in~\eqref{eq:eig-cond2}
  gives $b_1^\top A b_1 = \lambda (a_1^\top b_1 + a_2^\top
  b_2)$. These conditions imply that $a_1^\top A a_1 = - b_1^\top A
  b_1$. Since $A$ is negative semi-definite, we obtain $a_1 , b_1 \in
  \nll(A)$. Note that $a_1, b_1 \not = 0$, because otherwise it would
  mean that $a=b=0$. Further, using this fact in the first equations
  of~\eqref{eq:eig-cond1} and~\eqref{eq:eig-cond2}, respectively, we
  get
  \begin{align*}
    B a_2 = - \lambda b_1, \qquad B b_2 = \lambda a_1.
  \end{align*}
  That is, $a_1, b_1 \in \rge(B)$, a contradiction.
\end{proof}

The following example illustrates an application of the above results
to a nonconvex constrained optimization problem.

\begin{example}\longthmtitle{Saddle-point dynamics for nonconvex
    optimization}\label{ex:nonconvex-opt}
  {\rm Consider the following constrained optimization on $\real^3$,
    \begin{subequations}\label{eq:nonconvexopt}
      \begin{align}
        \mathrm{minimize} & \quad (\norm{x} - 1)^2,
        \\
        \text{subject to} & \quad x_3 = 0.5,
      \end{align}
    \end{subequations}
    where $x=(x_1,x_2,x_3)\in \real^3$. The optimizers are
    $\setdef{x \in \real^3}{x_3 = 0.5, x_1^2 + x_2^2 = 0.75}$.  The
    Lagrangian $\map{L}{\real^3 \times \real}{\real}$ is given by
    \begin{equation*}
      L(x,z) = (\norm{x} - 1)^2 + z(x_3 - 0.5),
    \end{equation*}
    and its set of saddle points is the one-dimensional manifold
    $\saddleset{L} = \setdef{(x,z) \in \real^3 \times \real}{x_3 =
      0.5, \, x_1^2 + x_2^2 = 0.75, \, z=0}$.
    The saddle-point dynamics of~$L$ takes the form
    \begin{subequations}\label{eq:Ldynamics}
      \begin{align}
        \dot x & = -2\Bigl(1 - \frac{1}{\norm{x}} \Bigr) x - [0, 0,
                 z]^\top,
        \\
        \dot z & = x_3 - 0.5.
      \end{align}
    \end{subequations}
    Note that $\saddleset{L}$ is nonconvex and that $L$ is nonconvex
    in its first argument on any neighborhood of any saddle point.
    Therefore, results that rely on the convexity-concavity properties
    of~$L$ are not applicable to establish the asymptotic convergence
    of~\eqref{eq:Ldynamics} to the set of saddle points.  This can,
    however, be established through
    Corollary~\ref{cr:localsaddlesetconv31} by observing that the
    Jacobian of $\SD$ at any point of $\saddleset{L}$ has $0$ as an
    eigenvalue with multiplicity one and the rest of the eigenvalues
    are not on the imaginary axis. To show this, consider
      $(\xo,\zo) \in \saddleset{L}$. Note that $D \SD (\xo,\zo)
      = \begin{bmatrix} -2\xo^\top \xo & -e_3 \\ e_3^\top & 0
      \end{bmatrix}$, where $e_3 = [0,0,1]^\top$. One can deduce from
      this that $v \in \nll(D \SD(\xo,\zo))$ if and only if $\xo^\top
      [v_1,v_2,v_3]^\top = 0$, $v_3 = 0$, and $v_4 = 0$. These three
      conditions define a one-dimensional space and so $0$ is an
      eigenvalue of $D \SD(\xo,\zo)$ with multiplicity $1$. To show
      that the rest of eigenvalues do not lie on the imaginary axis, we show
      that the hypotheses of Lemma~\ref{le:eigenvalue} are met. At any
      saddle point $(\xo,\zo)$, we have $\gradient_{zx} L(\xo,\zo) =
      e_3$ and $\gradient_{xx} L(\xo,\zo) = 2 \xo^\top \xo$. If $v \in
      \rge(\gradient_{zx} L(\xo,\zo)) \cap \nll(\gradient_{xx}
      L(\xo,\zo))$ then $v = [0,0,\lambda]^\top$, $\lambda \in \real$,
      and $\xo^\top v = 0$. Since $(\xo)_3 = 0.5$, we get $\lambda =
      0$ and hence, the hypotheses of Lemma~\ref{le:eigenvalue} are
      satisfied.
    Figure~\ref{fig:nonconvex-opt-example} illustrates
    in simulation the convergence of the trajectories to a saddle
    point.  The point of convergence depends on the initial
      condition.}\oprocend
  \begin{figure}[htb!]
    \centering
    \subfigure[$(x,z)$]{\includegraphics[width=.44\linewidth]{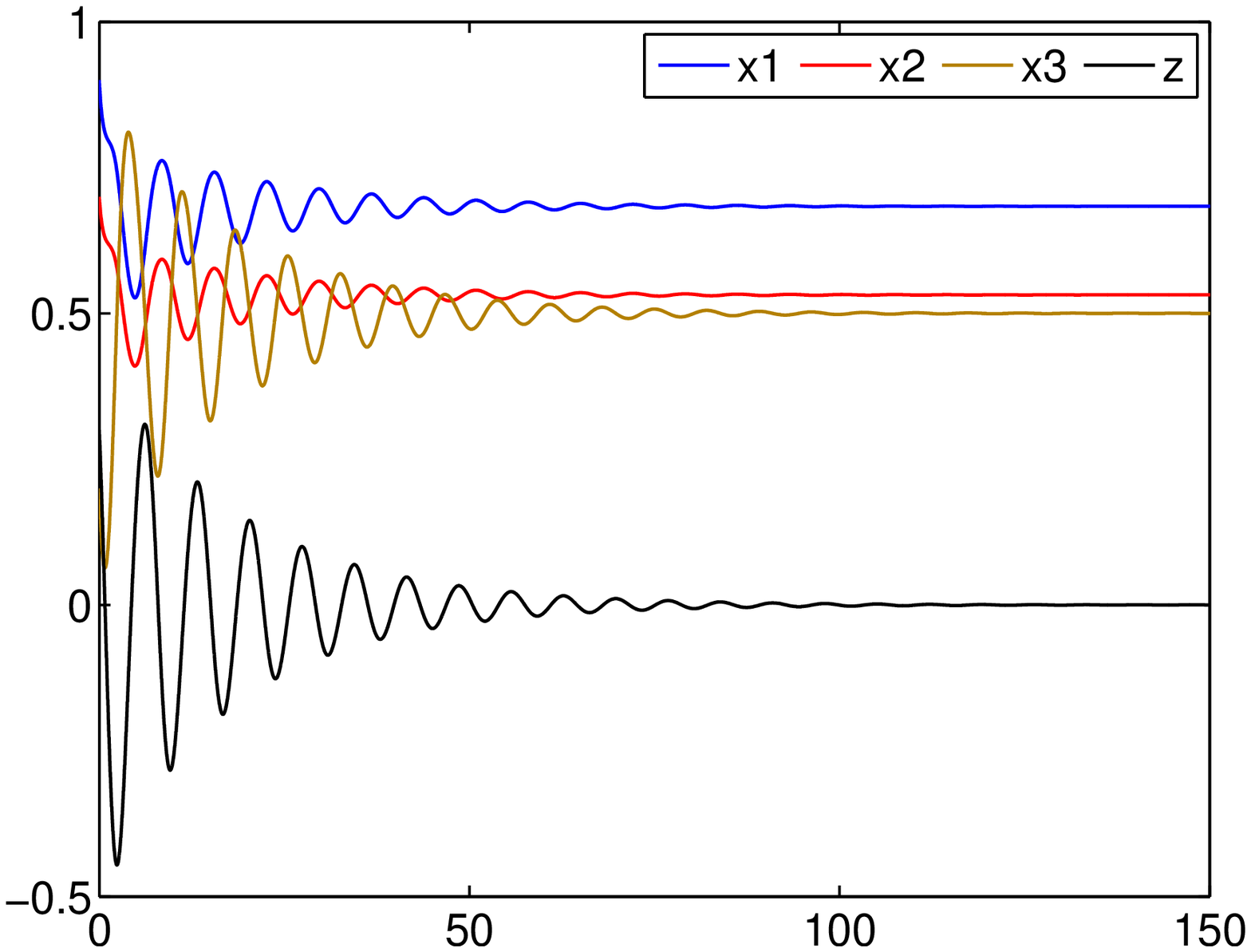}}
    \quad
    \subfigure[$(\norm{x} - 1)^2$]{\includegraphics[width=.44\linewidth]{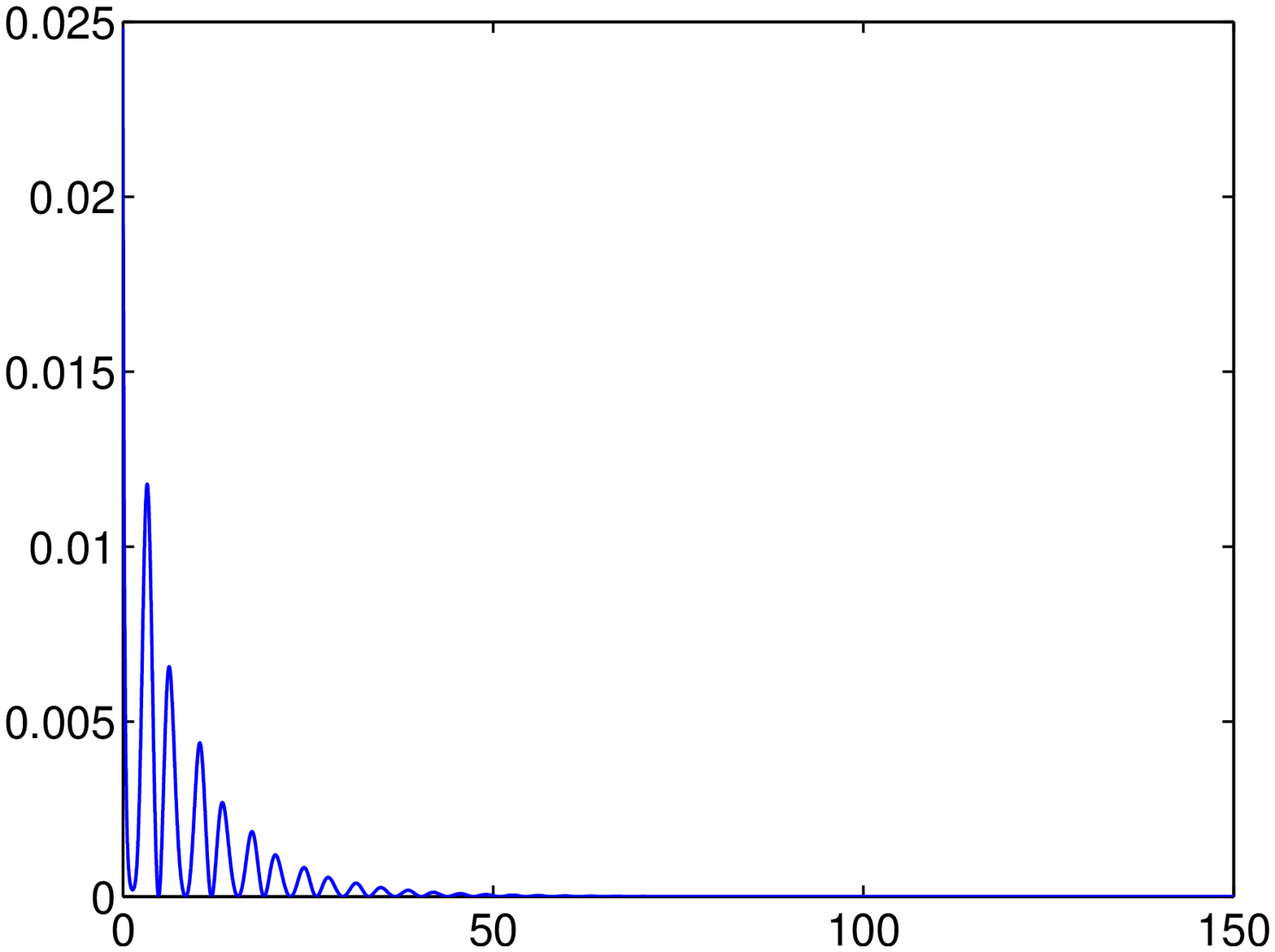}}
    \caption{(a) Trajectory of the saddle-point
      dynamics~\eqref{eq:Ldynamics} for the Lagrangian of the
      constrained optimization problem~\eqref{eq:nonconvexopt}. The
      initial condition is $(x,z) = (0.9,0.7,0.2,0.3)$.  The
      trajectory converges to $(0.68,0.53,0.50,0) \in \saddleset{L}$.
      (b) Evolution of the objective function of the
      optimization~\eqref{eq:nonconvexopt} along the trajectory.  The
      value converges to the
      minimum,~$0$.}\label{fig:nonconvex-opt-example}
    \vspace*{-1ex}
  \end{figure}
\end{example}

There are functions that do not satisfy the hypotheses of
Proposition~\ref{pr:localsaddlesetconv3} whose saddle-point dynamics
still seems to enjoy local asymptotic convergence properties. As an
example, consider the function $\map{F}{\real^2 \times \real}{\real}$,
\begin{align}\label{eq:asympex}
  F(x,z) = (\norm{x} - 1)^4 - z^2 \norm{x}^2 ,
\end{align}
whose set of saddle points is the one-dimensional manifold
$\saddleset{F} = \setdef{(x,z) \in \real^2 \times \real}{\norm{x}=1,
  z=0}$. The Jacobian of the saddle-point dynamics at any $(x,z) \in
\saddleset{F}$ has $-2$ as an eigenvalue and $0$ as the other
eigenvalue, with multiplicity $2$, which is greater than the dimension
of $\saddleset{F}$ (and therefore
Proposition~\ref{pr:localsaddlesetconv3} cannot be applied).
Simulations show that the trajectories of the saddle-point dynamics
asymptotically approach~$\saddleset{S}$ if the initial condition is
close enough to this set.  Our next result allows us to formally
establish this fact by studying the behavior of the distance function
along the proximal normals to~$\saddleset{F}$.
\begin{proposition}\longthmtitle{Asymptotic stability of manifold of
    saddle points via proximal normals}\label{pr:localsaddlesetconv4}
  Let $\map{F}{\real^n \times \real^m}{\real}$ be $\CC^2$ and
  $\SS \subset \saddleset{F}$ be a closed set.  Assume there exist
  constants $\lambda_M, k_1, k_2, \alpha_1, \beta_1 > 0$ and
  $L_x, L_z, \alpha_2, \beta_2 \ge 0$ such that the following hold
  \begin{enumerate}
    \item either $L_x = 0$ or $\alpha_1 \le \alpha_2 + 1$,
      \label{as:condition-one}
    \item either $L_z = 0$ or $\beta_1 \le \beta_2 + 1$,
      \label{as:condition-two}
    \item for every
  $(\xo,\zo) \in \SS$ and every proximal normal
  $\eta = (\eta_x,\eta_z) \in \real^n \times \real^m$ to $\SS$ at
  $(\xo,\zo)$ with $\norm{\eta} = 1$, the functions
  \begin{align*}
    & [0,\lambda_M) \ni \lambda \mapsto F(\xo+\lambda \eta_x, \zo) ,
    \\
    & [0,\lambda_M) \ni \lambda \mapsto F(\xo,\zo+\lambda \eta_z) ,
  \end{align*}
  are convex and concave, respectively, with
  \begin{subequations}\label{eq:Fproxbounds}
    \begin{align}
      F(\xo+\lambda \eta_x, \zo) - F(\xo,\zo) & \ge k_1 \norm{\lambda
        \eta_x}^{\alpha_1}, \label{eq:Fproxbounds-convex}
      \\
      F(\xo,\zo+\lambda \eta_z) - F(\xo,\zo) & \le - k_2 \norm{\lambda
        \eta_z}^{\beta_1}, \label{eq:Fproxbounds-concave}
    \end{align}
  \end{subequations}
  and, for all $\lambda \in [0,\lambda_M)$ and all $t \in [0,1]$,
  \begin{multline}\label{eq:Hessian-F-variation}
    \norm{\gradient_{xz} F(\xo+t\lambda \eta_x,\zo+\lambda \eta_z) -
      \gradient_{xz} F(\xo+\lambda \eta_x,\zo+t\lambda \eta_z)}
    \\
    \le L_x \norm{ \lambda \eta_x}^{\alpha_2} + L_z \norm{\lambda
      \eta_z}^{\beta_2} .
  \end{multline}
  \end{enumerate}
  Then, $\SS$ is locally asymptotically stable under the saddle-point
  dynamics $\SD$.  Moreover, the convergence of the trajectories is to
  a point if every point of $\SS$ is stable.  The convergence is global
  if, for every $\lambda_M \in \realnonnegative$, there exist
  $k_1,k_2,\alpha_1,\beta_1 >0$ such that the above hypotheses
  (i)-(iii) are satisfied by these constants along with $L_x = L_z =0$.
\end{proposition}
\begin{proof}
  Our proof is based on showing that there exists $\bar{\lambda} \in
  (0,\lambda_M]$ such that the distance function $d_{\SS}$ decreases
  monotonically and converges to zero along the trajectories of $\SD$
  that start in $\SS + B_{\bar{\lambda}}(0)$.
  From~\eqref{eq:gen-gradient-d},
  \begin{align*}
    \partial d_{\SS}^2(x,z) = \mathrm{co}\setdef{2(x-\xo; z-\zo)}
    { (\xo,\zo) \in \proj_{\SS}(x,z)} .
  \end{align*}
  Following~\cite{JC:08-csm-yo}, we compute the set-valued Lie
  derivative of $d_{\SS}^2$ along $\SD$, denoted
  $\setmap{\Lie_{\SD}d_{\SS}^2}{\real^n \times \real^m}{\real}$, as
  \begin{align*}
    \Lie_{\SD}d_{\SS}^2 (x,z) & = \mathrm{co}\setdef{-2(x-\xo)^\top
      \gradient_x F(x,z) +
      \\
      & 2(z - \zo)^\top \gradient_z F(x,z)}{(\xo,\zo) \in
      \proj_{\SS}(x,z)}.
  \end{align*}
  Since $d_{\SS}^2$ is globally Lipschitz and regular,
  cf. Section~\ref{sec:proximal-calculus}, the evolution of the
  function $d_{\SS}^2$ along any trajectory $t \mapsto (x(t),z(t))$
  of~\eqref{eq:saddledynamics} is differentiable at almost all $t \in
  \realnonnegative$, and furthermore, cf.~\cite[Proposition
  10]{JC:08-csm-yo},
  \begin{align*}
    \frac{d}{dt}(d_{\SS}^2(x(t),z(t)) \in \Lie_{\SD} d_{\SS}^2(x(t),z(t))
  \end{align*}
  for almost all $t \in \realnonnegative$.  Therefore, our goal is to
  show that $\max{\Lie_{\SD}d_{\SS}^2 (x,z)} < 0$ for all $(x,z) \in
  (\SS + B_{\bar{\lambda}}(0)) \setminus \SS$ for some $\bar{\lambda}
  \in (0,\lambda_M]$. Let $(x,z) \in \SS +
  B_{\lambda_M}(0)$ and take $(\xo,\zo) \in \proj_{\SS}(x,z)$. By
  definition, there exists a proximal normal $\eta = (\eta_x, \eta_z)$
  to $\SS$ at $(\xo,\zo)$ with $\norm{\eta} =1$ and $x = \xo + \lambda
  \eta_x$, $z = \zo + \lambda \eta_z$, and $\lambda \in
  [0,\lambda_M)$.  Let $2 \xi \in \Lie_{\SD}d_{\SS}^2(x,z)$ denote
  \begin{equation}\label{eq:xiexpression}
    \xi = - (x - \xo)^\top \gradient_x F(x,z) + (z - \zo)^\top
    \gradient_z F(x,z).
  \end{equation}
  Writing
  \begin{align*}
    \gradient_x F(x,z) & = \gradient_x F(x,\zo)
    + \int_0^1 \gradient_{zx} F(x,\zo+t(z-\zo)) (z - \zo) dt,
    \\
    \gradient_z F(x,z) & = \gradient_z F(\xo, z)
     + \int_0^1 \gradient_{xz} F(\xo+t(x-\xo),z) (x - \xo)dt,
  \end{align*} 
  and substituting in~\eqref{eq:xiexpression} we get 
  \begin{align}
    \xi & = - (x - \xo)^\top \gradient_x F(x,\zo) + (z - \zo)^\top
    \gradient_z F(\xo,z) + (z-\zo)^\top M
    (x-\xo), \label{eq:xiexpression2}
  \end{align} 
  where $M = \int_0^1 (\gradient_{xz} F(\xo+t(x-\xo),z) -
  \gradient_{xz} F(x,\zo+t(z-\zo))) dt$.  Using the convexity and
  concavity along the proximal normal and applying the
  bounds~\eqref{eq:Fproxbounds}, we obtain
  \begin{subequations}\label{eq:Fproxconvexbounds}
    \begin{align}
      - (x - \xo)^\top \gradient_x F(x,\zo) & \le F(\xo,\zo) - F(x,\zo)
      \le -k_1 \norm{\lambda \eta_x}^{\alpha_1},
      \\
      (z - \zo)^\top \gradient_z F(\xo,z) & \le F(\xo,z) - F(\xo,\zo)
       \le -k_2 \norm{\lambda \eta_z}^{\beta_1}.
    \end{align}
  \end{subequations}
  On the other hand, using~\eqref{eq:Hessian-F-variation}, we bound
  $M$ by
  \begin{equation}\label{eq:Mbound}
    \norm{M} \le L_x \norm{\lambda \eta_x}^{\alpha_2} + L_z
    \norm{\lambda \eta_z}^{\beta_2}.
  \end{equation}
  Using~\eqref{eq:Fproxconvexbounds} and~\eqref{eq:Mbound}
  in~\eqref{eq:xiexpression2}, and rearranging the terms yields
  \begin{align*} 
    \xi & \le \bigl(-k_1 \norm{\lambda \eta_x}^{\alpha_1} + L_x
    \norm{\lambda \eta_x}^{\alpha_2 +1} \norm{\lambda \eta_z} \bigr)
    + \bigl(- k_2 \norm{\lambda \eta_z}^{\beta_1} + L_z
    \norm{\lambda \eta_z}^{\beta_2 +1} \norm{\lambda \eta_x} \bigr).
  \end{align*}
  If $L_x = 0$, then the first parenthesis is negative whenever
  $\lambda \eta_x \not = 0$ (i.e., $x \not = \xo$). If $L_x \not = 0$
  and $\alpha_1 \le \alpha_2 + 1$, then for $\norm{\lambda \eta_x} <
  1$ and $\norm{\lambda \eta_z} < \min(1,k_1/L_x)$, the first
  parenthesis is negative whenever $\lambda \eta_x \not = 0$.
  Analogously, the second parenthesis is negative for $z \not = \zo$
  if either $L_z = 0$ or $\beta_1 \le \beta_2 +1$ with $\norm{\lambda
    \eta_z} < 1$ and $\norm{\lambda \eta_x} < \min(1,k_2/L_z)$.  Thus,
  if $\lambda < \min \{1,k_1/L_x,k_2/L_z\}$ (excluding from the $\min$
  operation the elements that are not well defined due to the
  denominator being zero), then
  hypotheses~\ref{as:condition-one}-\ref{as:condition-two} imply
  that $\xi < 0$ whenever $(x,z) \not = (\xo,\zo)$. Moreover, since
  $(\xo,\zo) \in \proj_{\SS}(x,z)$ was chosen arbitrarily, we conclude
  that $\max{\Lie_{\SD} d_{\SS}^2(x,z)} < 0$ for all $(x,z) \in \SS +
  B_{\bar{\lambda}}(0)$ where $\bar{\lambda} \in (0,\lambda_M]$
  satisfies $\bar{\lambda} < \min \{1,k_1/L_x,k_2/L_z\}$. This proves
  the local asymptotic stability. Finally, convergence to a point
  follows from Lemma~\ref{le:convtopoint} and global convergence
  follows from the analysis done above.
\end{proof}

Intuitively, the hypotheses of
Proposition~\ref{pr:localsaddlesetconv4} imply that along the proximal
normal to the saddle set, the convexity (resp.  concavity) in the
$x$-coordinate (resp. $z$-coordinate) is `stronger' than the influence
of the $x$- and $z$-dynamics on each other, represented by the
off-diagonal Hessian terms. When this coupling is absent (i.e.,
$\gradient_{xz} F \equiv 0$), the $x$- and $z$-dynamics are
independent of each other and they function as individually aiming to
minimize (resp. maximize) a function of one variable, thereby,
reaching a saddle point.  Note that the assumptions of
Proposition~\ref{pr:localsaddlesetconv4} do not imply that $F$ is
locally convex-concave. As an example, the function
in~\eqref{eq:asympex} is not convex-concave in any neighborhood of any
saddle point but we show next that it satisfies the assumptions of
Proposition~\ref{pr:localsaddlesetconv4}, establishing local
asymptotic convergence of the respective saddle-point dynamics.

\begin{example}\longthmtitle{Convergence analysis via  proximal
    normals}\label{ex:proxnormal} 
  {\rm Consider the function $F$ defined in~\eqref{eq:asympex}.
    Consider a saddle point $(\xo,\zo)= ( \cos \theta , \sin \theta,
    0) \in \saddleset{F}$, where $\theta \in [0,2\pi)$.  Let
  \begin{align*}
    \eta = (\eta_x,\eta_z) = ((a_1 \cos \theta, a_1 \sin \theta), a_2)
    ,
  \end{align*}
  with $a_1, a_2 \in \real$ and $a_1^2 + a_2^2 = 1$, be a proximal
  normal to $\saddleset{F}$ at $(\xo,\zo)$.  Note that the function
  $\lambda \mapsto F(\xo+\lambda \eta_x,\zo) = (\lambda a_1 )^4$ is
  convex, satisfying~\eqref{eq:Fproxbounds-convex} with $k_1 =1$ and
  $\alpha_1 = 4$. The function $\lambda \mapsto F(\xo, \zo+\lambda
  \eta_z) = -(\lambda a_2)^2$ is concave,
  satisfying~\eqref{eq:Fproxbounds-concave} with $k_2 =1$, $\beta_1 =
  2$. Also, given any $\lambda_M>0$ and for all $t \in [0,1]$, we can write
  \begin{align*}
    & \norm{\gradient_{xz} F(\xo+t\lambda \eta_x,\zo+\lambda \eta_z) -
      \gradient_{xz} F(\xo+\lambda \eta_x, \zo+t\lambda \eta_z)}
    \\
    & \qquad = \norm{-4(\lambda a_2) (1+t\lambda a_1)
      \left(\begin{smallmatrix} \cos \theta \\ \sin
          \theta \end{smallmatrix}\right) + 4 (t\lambda a_2)
      (1+\lambda a_1) \left(\begin{smallmatrix} \cos \theta \\ \sin
          \theta \end{smallmatrix}\right) },
    \\
    & \qquad \le \norm{4 (\lambda a_2) (1+t\lambda a_1) - 4(t \lambda
      a_2)(1+\lambda a_1)},
    \\
    & \qquad \le 8 (1+\lambda a_1) (\lambda a_2) \le L_z (\lambda
    a_2),
  \end{align*}
  for $\lambda \le \lambda_M$, where $L_z = 8 (1+\lambda_M a_1)$. This
  implies that $L_x = 0$, $L_z \not = 0$ and $\beta_2 = 1$.  Therefore,
  hypotheses (i)-(iii) of Proposition~\ref{pr:localsaddlesetconv4}
  are satisfied and this establishes asymptotic convergence of the
  saddle-point dynamics.  Figure~\ref{fig:add-case} illustrates this
  fact. Note that since $L_z \not = 0$, we cannot guarantee global
  convergence. } \oprocend
  \begin{figure}[htb!]
    \centering
    \subfigure[$(x,z)$]{\includegraphics[width=.44\linewidth]{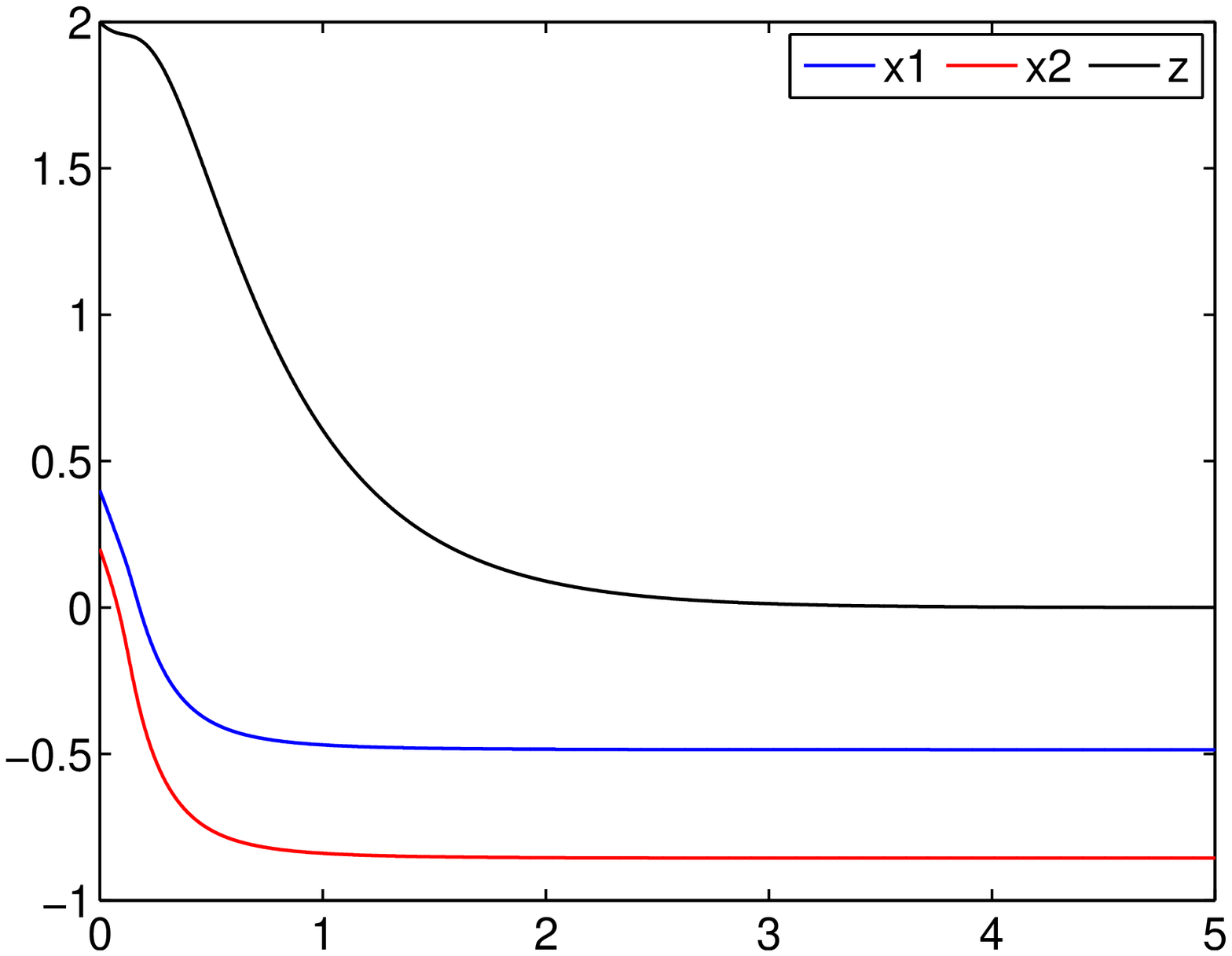}}
    \quad
    \subfigure[$F$]{\includegraphics[width=.44\linewidth]{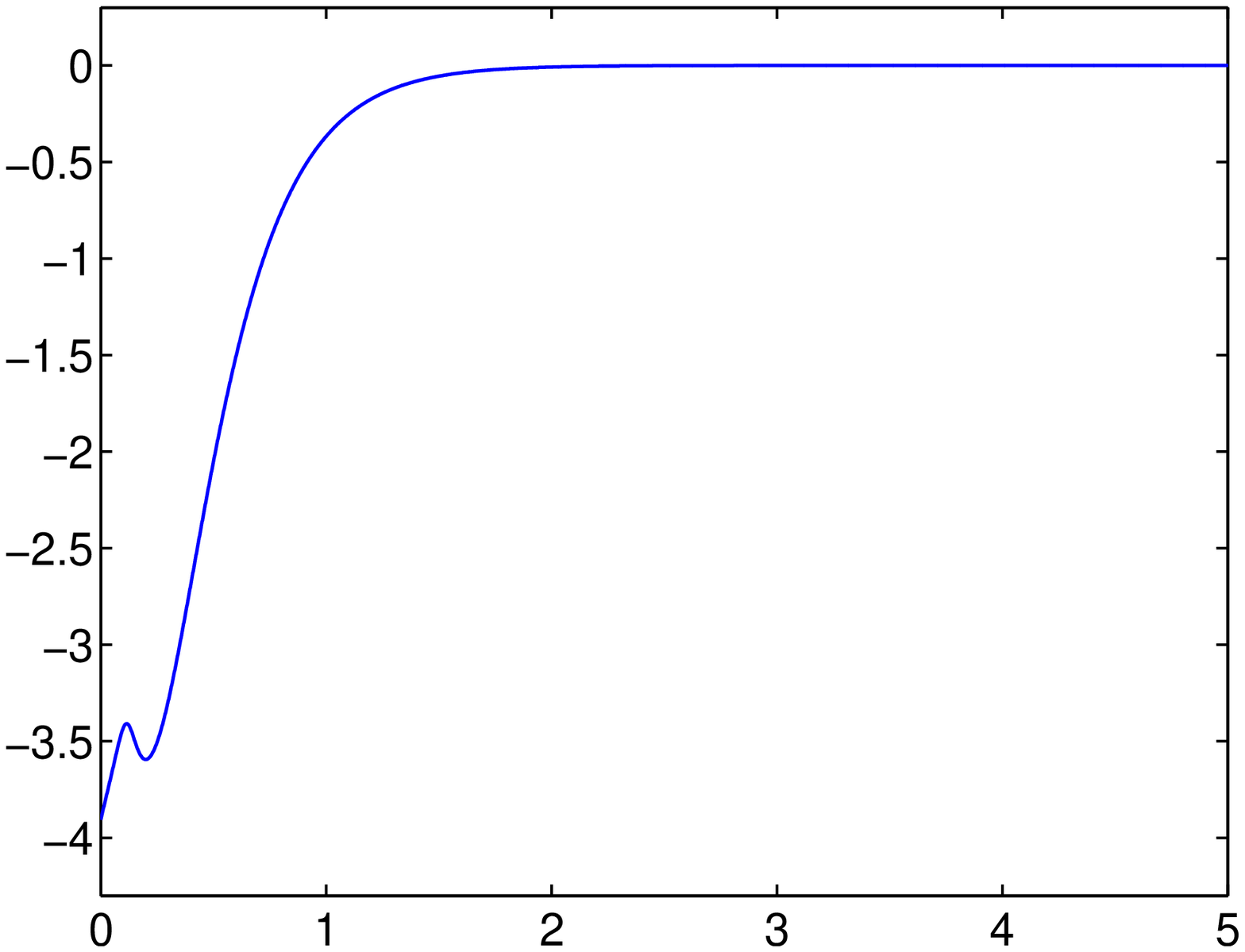}}
    \caption{(a) Trajectory of the saddle-point dynamics for the
      function defined by~\eqref{eq:asympex}. The initial condition is
      $(x,z) = (0.1,0.2,4)$.  The trajectory converges to
      $(0.49,0.86,0) \in \saddleset{F}$. (b) Evolution of the
      function~$F$ along the trajectory.  The value converges to $0$,
      the value that the function takes on its saddle
      set. }\label{fig:add-case}
    \vspace*{-1ex}
  \end{figure}
\end{example}

Interestingly, Propositions~\ref{pr:localsaddlesetconv3}
and~\ref{pr:localsaddlesetconv4} complement each other. The
function~\eqref{eq:asympex} satisfies the hypotheses of
Proposition~\ref{pr:localsaddlesetconv4} but not those of
Proposition~\ref{pr:localsaddlesetconv3}. Conversely, the Lagrangian
of the constrained optimization~\eqref{eq:nonconvexopt} satisfies the
hypotheses of Proposition~\ref{pr:localsaddlesetconv3} but not those
of Proposition~\ref{pr:localsaddlesetconv4}.

In the next result, we consider yet another scenario where the saddle
function might not be convex-concave in its arguments but the
saddle-point dynamics converges to the set of equilibrium points. As a
motivation, consider the function $\map{F}{\real \times
  \real}{\real}$,
$F(x,z) = xz^2$.
The set of saddle points of $F$ are $\saddleset{F} = \realnonpositive
\times \{0\}$. One can show that, at the saddle point $(0,0)$, neither
the hypotheses of Proposition~\ref{pr:localsaddlesetconv3} nor those
of Proposition~\ref{pr:localsaddlesetconv4} are satisfied.  Yet,
simulations show that the trajectories of the dynamics converge to the
saddle points from almost all initial conditions in $\real^2$, see
Figure~\ref{fig:nonconvex-linear-case} below. This asymptotic behavior
can be characterized through the following result which
generalizes~\cite[Theorem~3]{VAK-YSP:87}.

\begin{proposition}\longthmtitle{Global asymptotic stability of
    equilibria of saddle-point dynamics for saddle functions linear in
    one argument}\label{pr:nonconvex-linear-x}
  For $\map{F}{\real^n \times \real^m}{\real}$, assume the following
  form $F(x,z) = g(z)^\top x$, where $\map{g}{\real^m}{\real^n}$ is
  $\CC^1$.  Assume that there exists $(\xo,\zo) \in \saddleset{F}$
  such that
  \begin{enumerate}
  \item $F(\xo,\zo) \ge F(\xo,z)$ for all $z \in
    \real^m$, 
    \label{as:nonconvex-linear-1}
  \item for any $z \in \real^m$, the condition $g(z)^\top \xo = 0$
    implies $g(z) = 0$, \label{as:nonconvex-linear-2}
  \item any trajectory of $\SD$ is bounded.
    \label{as:nonconvex-linear-3}
  \end{enumerate}
  Then, all trajectories of the saddle-point dynamics $\SD$ converge
  asymptotically to the set of equilibrium points of $\SD$.
\end{proposition}
\begin{proof}
  Consider the function $\map{V}{\real^n \times \real^m}{\real}$,
  \begin{equation*}
    V(x,z) = -\xo^\top x.
  \end{equation*}
  The Lie derivative of $V$ along the saddle-point dynamics $\SD$ is  
  \begin{align}\label{eq:nonconvex-linear-evol-v}
    \Lie_{\SD}V(x,z) =  \xo^\top \gradient_x F(x,z) = \xo^\top g(z) =
    F(\xo,z) \le F(\xo,\zo) = 0,
  \end{align}
  where in the inequality we have used
  assumption~\ref{as:nonconvex-linear-1}, and $F(\xo,\zo) = 0$ is
  implied by the definition of the saddle point, that is, $\gradient_x
  F(\xo,\zo) = g(\zo) = 0$.  Now consider any trajectory $t \mapsto
  (x(t),z(t))$, $(x(0),z(0)) \in \real^n \times \real^m$ of
  $\SD$. Since the trajectory is bounded by
  assumption~\ref{as:nonconvex-linear-3}, the application of the
  LaSalle Invariance Principle~\cite[Theorem 4.4]{HKK:02} yields that
  the trajectory converges to the largest invariant set $\MM$
  contained in $\setdef{(x,z) \in \real^n \times
    \real^m}{\Lie_{\SD}V(x,z) = 0}$, which
  from~\eqref{eq:nonconvex-linear-evol-v} is equal to the set
  $\setdef{(x,z) \in \real^n \times \real^m}{F(\xo,z) = 0}$. Let
  $(x,z) \in \MM$. Then, we have $F(\xo,z) = g(z)^\top \xo = 0$ and by
  hypotheses~\ref{as:nonconvex-linear-2} we get $g(z) = 0$. Therefore,
  if $(x,z) \in \MM$ then $g(z) = 0$.  Consider the trajectory $t
  \mapsto (x(t),z(t))$ of $\SD$ with $(x(0),z(0)) = (x,z)$ which is
  contained in $\MM$. Then, along the trajectory we have
  \begin{align*}
    \dot x(t) = - \gradient_x F(x(t),z(t)) = - g(z(t)) = 0
  \end{align*}
  Further, note that along this trajectory we have
  $g(z(t)) = 0$ for all $t \ge 0$. Thus, $\frac{d}{dt} g(z(t))
  = 0$ for all $t \ge 0$, which implies that
  \begin{align*}
    \frac{d}{dt} g(z(t)) = Dg(z(t)) \dot z(t) = Dg(z(t)) Dg(z(t))^\top x =
    0.
  \end{align*}
  From the above expression we deduce that $\dot z (t) = Dg(z(t))^\top x = 0$.
  This can be seen from the
  fact that $Dg(z(t)) Dg(z(t))^\top x = 0$ implies $x^\top Dg(z(t))
  Dg(z(t))^\top x = (Dg(z(t))^\top x)^2 = 0$.
  From the above reasoning, we conclude that $(x,z)$ is an
  equilibrium point of $\SD$.  
\end{proof}

The proof of Proposition~\ref{pr:nonconvex-linear-x} hints at the fact
that hypothesis (ii) can be omitted if information about other saddle
points of $F$ is known. Specifically, consider the case where $n$
saddle points $(\xo^{(1)},\zo^{(1)}), \dots , (\xo^{(n)},\zo^{(n)})$
of $F$ exist, each satisfying hypothesis (i) of
Proposition~\ref{pr:nonconvex-linear-x} and such that the vectors
$\xo^{(1)}, \dots, \xo^{(n)}$ are linearly independent. In this
scenario, for those points $z \in \real^m$ such that $g(z)^\top
\xo^{(i)} = 0$ for all $i \in \until{n}$ (as would be obtained in the
proof), the linear independence of $\xo^{(i)}$'s already implies that
$g(z) = 0$, making hypothesis (ii) unnecessary.

\begin{corollary}\longthmtitle{Almost global asymptotic stability of
    saddle points for saddle functions linear in
     one  argument}\label{cr:nonconvex-linear-x}
   If, in addition to the hypotheses of
   Proposition~\ref{pr:nonconvex-linear-x}, the set of equilibria of
   $\SD$ other than those belonging to $\saddleset{F}$ are unstable,
   then the trajectories of $\SD$ converge asymptotically to
   $\saddleset{F}$ from almost all initial conditions (all but the
   unstable equilibria). Moreover, if each point in $\saddleset{F}$ is
   stable under~$\SD$, then $\saddleset{F}$ is almost globally
   asymptotically stable under the saddle-point dynamics~$\SD$ and the
   trajectories converge to a point in $\saddleset{F}$.
\end{corollary}

Next, we illustrate how the above result can be applied to the
motivating example given before
Proposition~\ref{pr:nonconvex-linear-x} to infer almost global
convergence of the trajectories.

\begin{example}\longthmtitle{Convergence for saddle functions linear
    in  one
    argument}\label{ex:nonconvex-linear} {\rm Consider again $F(x,z) =
    xz^2$ with $\saddleset{F} = \setdef{(x,z) \in \real \times
      \real}{x \le 0 \text{ and } z =0}$.  Pick $(\xo,\zo) =
    (-1,0)$. One can verify that this saddle point satisfies the
    hypotheses~\ref{as:nonconvex-linear-1}
    and~\ref{as:nonconvex-linear-2} of
    Proposition~\ref{pr:nonconvex-linear-x}. Moreover, along any
    trajectory of the saddle-point dynamics for $F$, the function $x^2
    + \frac{z^2}{2}$ is preserved, which implies that all trajectories
    are bounded. One  can also see  that the equilibria of the
    saddle-point dynamics that are not saddle points, that is the set
    $\realpositive \times \{0\}$,
    are unstable.  Therefore, from
    Corollary~\ref{cr:nonconvex-linear-x}, we conclude that the
    trajectories of the saddle-point dynamics asymptotically converge
    to the set of saddle points from almost all initial conditions. 
    Figure~\ref{fig:nonconvex-linear-case} illustrates these
    observations.
   \begin{figure}[htb!]
    \centering
    \subfigure[$(x,z)$]{\includegraphics[width=.3\linewidth]{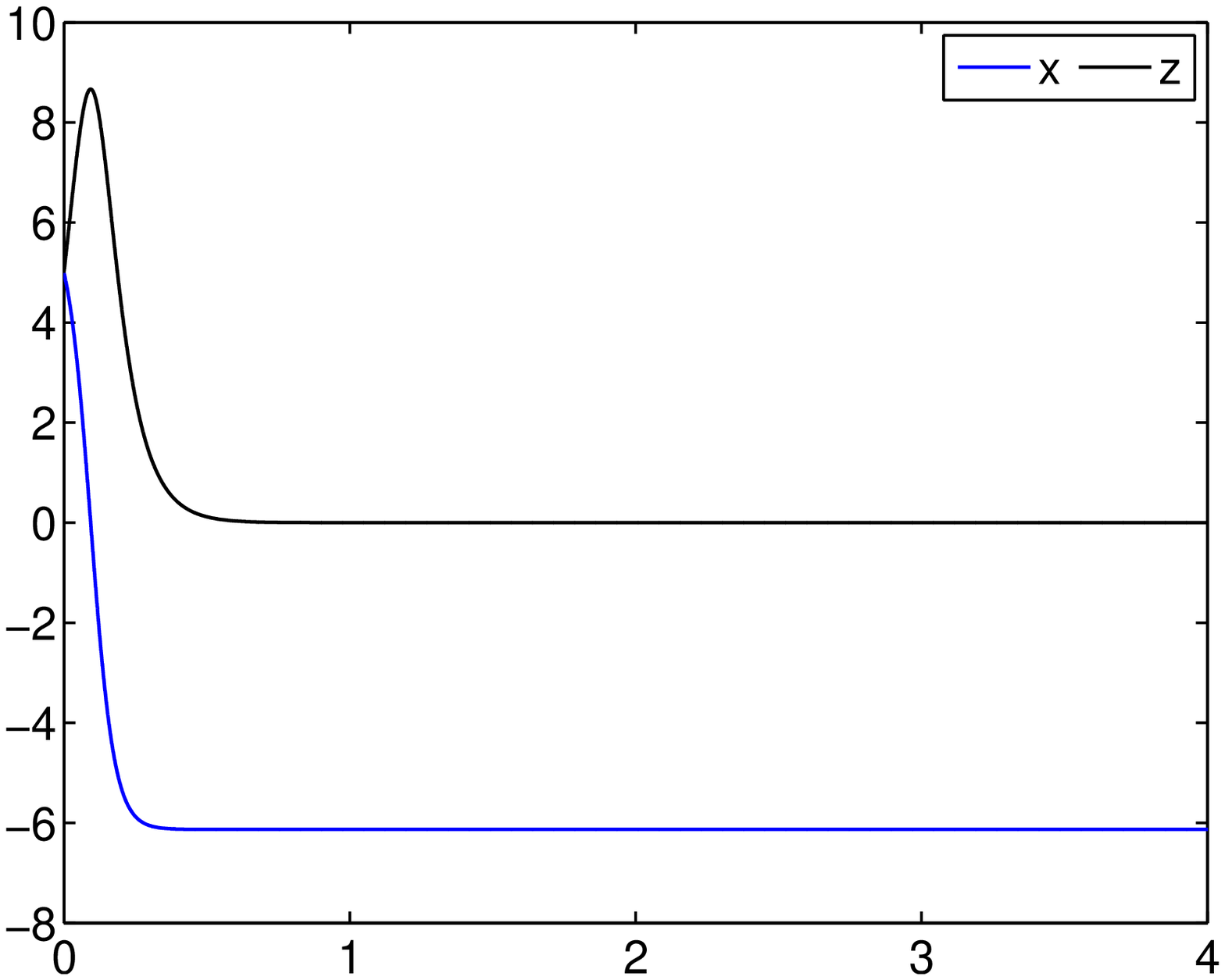}}
    \quad
    \subfigure[$F$]{\includegraphics[width=.3\linewidth]{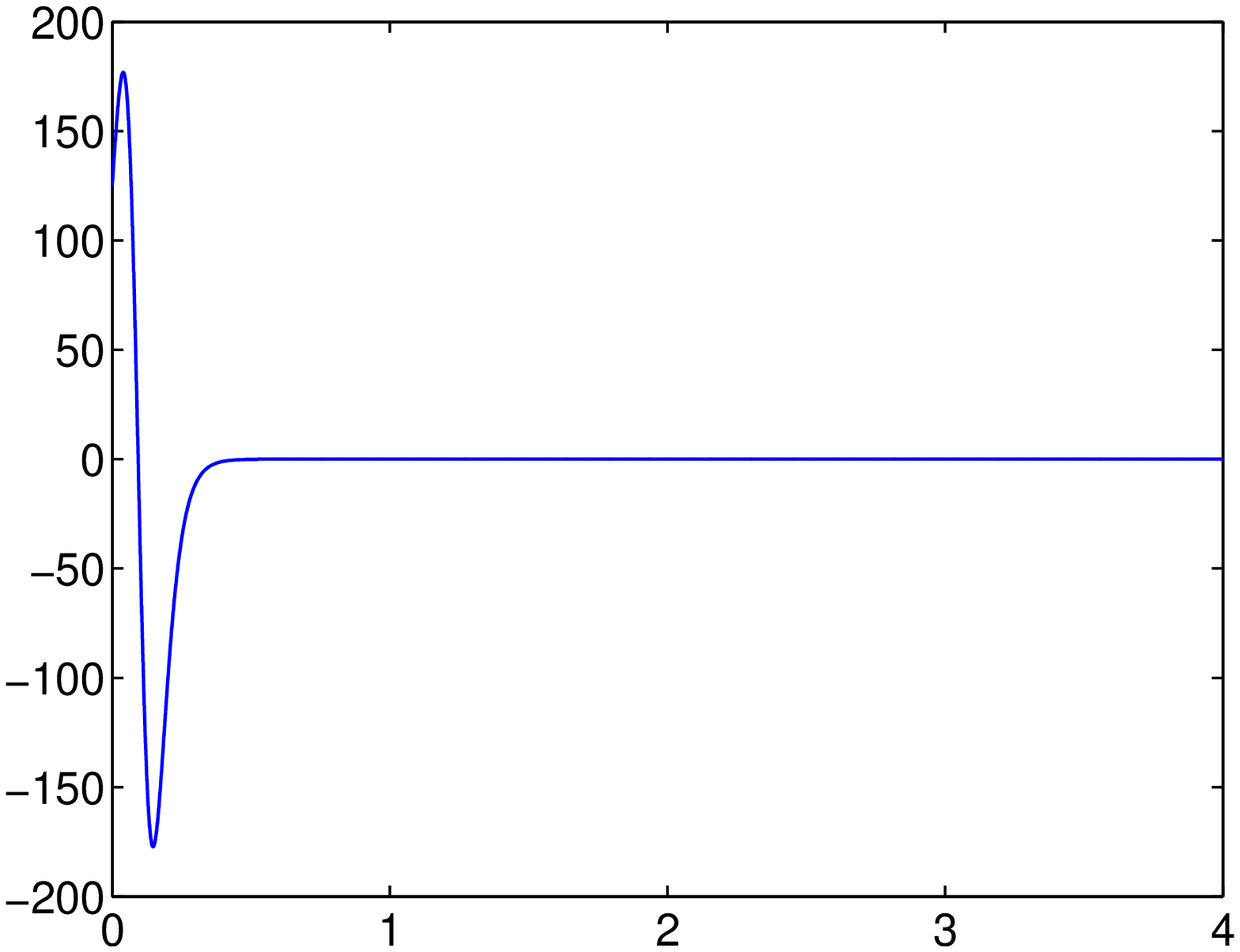}}
    \quad
    \subfigure[Vector field
    $\SD$]{\includegraphics[width=.25\linewidth]{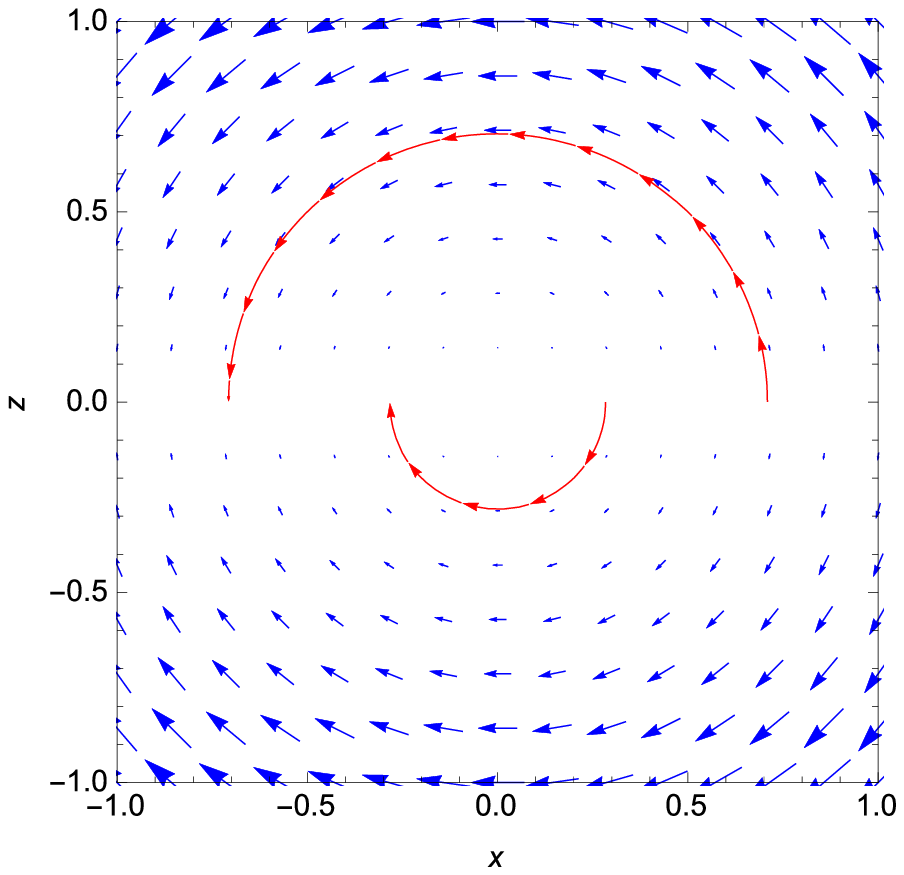}}
    \caption{(a) Trajectory of the saddle-point dynamics for the
      function $F(x,z) = xz^2$. The initial condition is
      $(x,z) = (5,5)$.  The trajectory converges to
      $(-6.13,0) \in \saddleset{F}$. (b) Evolution of the
      function~$F$ along the trajectory.  The value converges to $0$,
      the value that the function takes on its saddle
      set. (c) The vector field $\SD$, depicting that the set of saddle
      points are attractive while the other equilibrium points
      $\realpositive \times \{0\}$
      are unstable.}\label{fig:nonconvex-linear-case}
    \vspace*{-1ex}
  \end{figure}
}
  \oprocend
\end{example}

\section{Conclusions}\label{sec:conclusions}

We have studied the asymptotic stability of the saddle-point dynamics
associated to a continuously differentiable function. We have identified
a set of complementary conditions under which the trajectories of the
dynamics are proved to converge to the set of saddle points of the
saddle function and, wherever feasible, we have also established global
stability guarantees and convergence to a point in the set. Our first
class of convergence results is based on the convexity-concavity
properties of the saddle function.  In the absence of these properties,
our second class of results explore, respectively, the existence of
convergence guarantees using linearization techniques, the properties of
the saddle function along proximal normals to the set of saddle points,
and the linearity properties of the saddle function in one variable. For
the linearization result, borrowing ideas from center manifold theory,
we have established a general stability result of a manifold of
equilibria for a piecewise twice continuously differentiable vector
field.  Several examples throughout the paper highlight the connections
among the results and illustrate their applicability, in particular, for
finding the primal-dual solutions of constrained optimization problems.
Future work will study the robustness properties of the dynamics against
disturbances, investigate the characterization of the rate of
convergence, generalize the results to the case of nonsmooth functions
(where the associated saddle-point dynamics takes the form of a
differential inclusion involving the generalized gradient of the
function), and explore the application to optimization problems with
inequality constraints. We also plan to build on our results to
synthesize distributed algorithmic solutions for various networked
optimization problems in power networks.

\section{Acknowledgements}
Ashish Cherukuri and Jorge Corte\'{e}s's research was partially
supported by NSF award ECCS-1307176. Bahman Gharesifard's research was
supported by an NSERC Discovery grant.

\section*{Appendix}
\renewcommand{\theequation}{A.\arabic{equation}}
\renewcommand{\thetheorem}{A.\arabic{theorem}}

This section contains some auxiliary results for our convergence
analysis in Sections~\ref{sec:case1} and~\ref{sec:case2}. Our
  first result establishes the constant value of the saddle function
  over its set of (local) saddle points.

\begin{lemma}\longthmtitle{Constant function value over saddle 
    points}\label{le:Fconstant}
  For $\map{F}{\real^n \times \real^m}{\real}$ continuously
  differentiable, let $\SS \subset \saddleset{F}$ be a path connected
  set. If $F$ is locally convex-concave on $\SS$, then $F_{|\SS}$ is
  constant.
\end{lemma}
\begin{proof}
  We start by considering the case when $\SS$ is compact. Given $(x,z)
  \in \SS$, let $\delta(x,z) > 0$ be such that $B_{\delta(x,z)}(x,z)
  \subset (\UU_{x} \times \UU_{z}) \cap \UU$, where $\UU_x$ and
  $\UU_z$ are neighborhoods where the saddle
  property~\eqref{eq:saddleinequality} holds and $\UU$ is the
  neighborhood of $(x,z)$ where local convexity-concavity holds (cf.
  Section~\ref{sec:saddle-points}). This defines a covering of $\SS$
  by open sets as
  \begin{align*}
    \SS \subset \cup_{(x,z) \in \SS} B_{\delta(x,z)}(x,z) .
  \end{align*}
  Since $\SS$ is compact, there exist a finite number of points
  $(x_1,z_1), (x_2,z_2), \dots, (x_n,z_n)$ in $\SS$ such that
  $\cup_{i=1}^n B_{\delta(x_i,z_i)}(x_i,z_i)$ covers $\SS$.  For
  convenience, denote $B_{\delta(x_i,z_i)}(x_i,z_i)$ by $B_i$.  Next,
  we show that $F_{|\SS \cap B_i}$ is constant for all
  $i \in \until{n}$.  To see this, let $(\xb,\zb) \in \SS \cap
  B_i$. From~\eqref{eq:saddleinequality}, we have
  \begin{equation}\label{eq:saddlexz}
    F(x_i,\zb) \le F(x_i,z_i) \le F(\xb,z_i).
  \end{equation}
  From the convexity of $x \mapsto F(x,\zb)$ over
  $\UU \cap (\real^n \times \setr{\zb})$, (cf. definition of local
  convexity-concavity in Section~\ref{sec:saddle-points}), and the
  fact that $\gradient_x F(\xb,\zb) = 0$, we obtain
  $F(x_i,\zb) \ge F(\xb,\zb) + (x_i - \xb)^\top \gradient_x F(\xb,\zb)
  = F(\xb,\zb)$.
  Similarly, using the concavity of $z \mapsto F(\xb,z)$, we get
  $F(\xb,z_i) \le F(\xb,\zb)$. These inequalities together
  with~\eqref{eq:saddlexz} yield
  \begin{align*}
    F(x_i,z_i) \le F(\xb,z_i) \le F(\xb,\zb) \le F(x_i,\zb) \le
    F(x_i,z_i).
  \end{align*}
  That is, $F(\xb,\zb) = F(x_i,z_i)$ and hence $F_{|\SS \cap B_i}$ is
  constant. Using this reasoning, if $\SS \cap B_i \cap B_j \not =
  \emptyset$ for any $i,j \in \until{n}$, then $F_{|\SS \cap (B_i \cup
    B_j)}$ is constant.  Using that $\SS$ is path connected, the
  fact~\cite[p. 117]{JD:66} states that, for any two points
  $(x_l,z_l), (x_m,z_m) \in \SS$, there exist distinct members $i_1,
  i_2, \dots, i_k$ of the set $\until{n}$ such that $(x_l,z_l) \in \SS
  \cap B_{i_1}$, $(x_m,z_m) \in \SS \cap B_{i_k}$ and $\SS \cap
  B_{i_t} \cap B_{i_{t+1}} \not = \emptyset$ for all $t \in
  \until{k-1}$. Hence, we conclude that $F_{|\SS}$ is constant.
  Finally, in the case when $\SS$ is not compact, pick any two points
  $(x_l,z_l), (x_m,z_m) \in \SS$ and let $\gamma:[0,1] \to \SS$ be a
  continuous map with $\gamma(0) = (x_l,z_l)$ and $\gamma(1) =
  (x_m,z_m)$ denoting the path between these points.  The image
  $\gamma([0,1]) \subset \SS$ is closed and bounded, hence compact,
  and therefore, $F_{|\gamma([0,1])}$ is constant. Since the two
  points are arbitrary, we conclude that $F_{|\SS}$ is constant.
\end{proof}

The difficulty in Lemma~\ref{le:Fconstant} arises due to the local
nature of the saddle points (the result is instead straightforward for
global saddle points). The next result provides a first-order condition
for strongly quasiconvex functions.

\begin{lemma}\longthmtitle{First-order property of a strongly
    quasiconvex function}\label{le:first-quasi}
  Let $\map{f}{\real^n}{\real}$ be a $\CC^1$ function that is strongly
  quasiconvex on a convex set $\DD \subset \real^n$.  Then, there
  exists a constant $s > 0$ such that 
  \begin{align}\label{eq:f-order-cond}
    f(x) \le f(y) \Rightarrow  \gradient f(y)^\top (x-y) \le -s 
    \norm{x-y}^2,
  \end{align}
  for any  $x, y \in \DD$.
\end{lemma}
\begin{proof}
  Consider $x,y \in \DD$ such that $f(x) \le f(y)$. From strong
  quasiconvexity we have $f(y) \ge f(\lambda x + (1-\lambda)y) + s
  \lambda (1-\lambda) \norm{x-y}^2$, for any $\lambda \in
  [0,1]$. Rearranging,
  \begin{equation}\label{eq:f-ineq}
    f(\lambda x + (1-\lambda)y) - f(y) \le - s \lambda (1-\lambda)
    \norm{x-y}^2.
  \end{equation}
  On the other hand, the Taylor's approximation of $f$ at $y$ yields
  the following equality at point
  $y + \lambda(x-y)$, which is equal to $\lambda x + (1-\lambda)y$, as
  \begin{align}\label{eq:f-taylor}
    f(\lambda x +(1-\lambda) y) - f(y) & = \gradient f(y)^\top
    (\lambda x +(1-\lambda)y -y) + g(\lambda x + (1-\lambda) y - y)
    \notag
    \\
    & = \lambda \gradient f(y)^\top (x - y) + g(\lambda(x-y)) ,
  \end{align}
  for some function $g$ with the property
  $\lim_{\lambda \to 0} \frac{g(\lambda(x-y))}{\lambda} = 0$.
  Using~\eqref{eq:f-taylor} in~\eqref{eq:f-ineq}, dividing by
  $\lambda$, and taking the limit $\lambda \to 0$ yields the result.
\end{proof}

The next result is helpful when dealing with dynamical systems that
have non-isolated equilibria to establish the asymptotic convergence
of the trajectories to a point, rather than to a set.

\begin{lemma}\longthmtitle{Asymptotic convergence to a
    point~\cite[Corollary 5.2]{SPB-DSB:03}}\label{le:convtopoint}
  Consider the nonlinear system
  \begin{equation}\label{eq:nonlinearsys}
    \dot x(t) = f(x(t)), \quad x(0) = x_0, 
  \end{equation}
  where $f:\real^n \to \real^n$ is locally Lipschitz.  Let
  $\WW \subset \real^n$ be a compact set that is positively invariant
  under~\eqref{eq:nonlinearsys} and let $\EE \subset \WW$ be a set of
  stable equilibria. If a trajectory $t \mapsto x(t)$
  of~\eqref{eq:nonlinearsys} with $x_0 \in \WW$ satisfies
  $\lim_{t \to \infty} d_{\EE}(x(t)) = 0$, then the trajectory
  converges to a point in $\EE$.
\end{lemma}

Finally, we establish the asymptotic stability of a manifold of
equilibria through linearization techniques. We start with a useful
intermediary result.
\begin{lemma}\longthmtitle{Limit points of Jacobian of a
  piecewise $\CC^2$ function}\label{le:limit-jacobian}
  Let $\map{f}{\real^n}{\real^n}$ be piecewise $\CC^2$. Then, for every
  $x \in \real^n$, there exists a finite index set $\II_x \subset
  \integerspositive$ and a set of matrices $\{A_{x,i} \in \real^{n
  \times n} \}_{i \in \II_x}$ such that
  \begin{equation}\label{eq:limit-set} 
    \setdef{A_{x,i}}{i \in \II_x} = \setdef{\lim_{k \to \infty}
    Df(x_k)} {x_k \to x, x_k \in \real^n \setminus \Omega_f}, 
  \end{equation} 
  where $\Omega_f$ is the set of points where $f$ is not
  differentiable. 
\end{lemma}
\begin{proof}
  Since $f$ is piecewise $\CC^2$, cf. Section~\ref{sec:notation}, let
  $\DD_1, \dots, \DD_m \subset \real^n$ be the finite collection of
  disjoint open sets such that $f$ is $\CC^2$ on $\DD_i$ for each $i
  \in \until{m}$ and $\real^n = \cup_{i=1}^m \mathrm{cl}(\DD_i)$.  Let
  $x \in \real^n$ and define $\II_x = \setdef{i \in \until{m}}{x \in
    \mathrm{cl}(\DD_i)}$ and $A_{x,i} = \setdef{\lim_{k \to \infty}
    Df(x_k)}{x_k \to x, x_k \in \DD_i}$.  Note that $A_{x,i}$ is
  uniquely defined for each $i$ as, by definition,
  $f_{|\mathrm{cl}(\DD_i)}$ is $\CC^2$. To show
  that~\eqref{eq:limit-set} holds for the above defined matrices,
  first note that the set $\setdef{A_{x,i}}{i \in \II_x}$ is included
  in the right hand side of~\eqref{eq:limit-set} by definition. 
  To show the other inclusion, consider any sequence
  $\{x_k\}_{k=1}^\infty \subset \real^n \setminus \Omega_f$ with $x_k
  \to x$. One can partition this sequence into subsequences, each
  contained in one of the sets $\DD_i$, $i \in \II_x$ and each
  converging to $x$. Thus, the limit $\lim_{k \to \infty}
  Df(x_k)$ is contained in the set $\{A_{x,i}\}_{i \in \II_x}$,
  proving the other inclusion and 
  yielding~\eqref{eq:limit-set}. Note that, in the nonsmooth analysis
  literature~\cite[Chapter 2]{FHC:83}, the convex hull of 
  matrices $\{A_{x,i}\}_{i \in \II_x}$ is known as the
  generalized Jacobian of $f$ at $x$.
\end{proof}

The following statement is an extension of~\cite[Exercise 6]{DH:81} to
vector fields that are only piecewise twice continuously
differentiable. Its proof is inspired, but cannot be directly implied
from, center manifold theory~\cite{JC:82}.

\begin{proposition}\longthmtitle{Asymptotic stability of a manifold of
    equilibrium points for piecewise $\CC^2$ vector
    fields}\label{pr:convtoman-cl}
  Consider the system
  \begin{equation}\label{eq:fsystem-cl}
    \dot x = f(x), 
  \end{equation}
  where $f: \real^n \to \real^n$ is piecewise $\CC^2$ and locally
  Lipschitz in a neighborhood of a $p$-dimensional submanifold of
  equilibrium points $\EE \subset \real^n$
  of~\eqref{eq:fsystem-cl}. 
  Assume that at each $\xo \in \EE$, the set of matrices $\{A_{\xo,i}
  \}_{i \in \II_{\xo}}$ from Lemma~\ref{le:limit-jacobian} satisfy:
  \begin{enumerate}
  \item 
    there exists an orthogonal matrix
    $Q \in \real^{n \times n}$ such that, for all $i \in \II_{\xo}$,
    \begin{align*}
      Q^\top A_{\xo,i} Q =
      \begin{bmatrix} 0 & 0
        \\
        0 & \tilde{A}_{\xo,i}
      \end{bmatrix},
    \end{align*}
    where $\tilde{A}_{\xo,i} \in \real^{n-p \times n-p}$,
  \item the eigenvalues of the matrices
    $\{\tilde{A}_{\xo,i}\}_{i \in \II_{\xo}}$ have negative real parts,
  \item there exists a positive definite matrix
    $P \in \real^{n-p \times n-p}$ such that
    \begin{align*}
      \tilde{A}_{\xo,i}^\top P + P \tilde{A}_{\xo,i} \prec 0,  
      \quad  \text{ for all } i \in \II_{(\xo,\zo)}.
    \end{align*}
  \end{enumerate}
  Then, $\EE$ is locally asymptotically stable
  under~\eqref{eq:fsystem-cl} and the trajectories converge to a point
  in $\EE$.
\end{proposition}
\begin{proof}
  Our strategy to prove the result is to linearize the vector field
  $f$ on each of the patches around any equilibrium point and employ a
  common Lyapunov function and a common upper bound on the growth of
  the second-order term to establish the convergence of the
  trajectories. This approach is an extension of the proof
  of~\cite[Theorem 8.2]{HKK:02}, where the vector field $f$ is assumed
  to be $\CC^2$ everywhere. Let $\xo \in \EE$. For convenience,
  translate $\xo$ to the origin of~\eqref{eq:fsystem-cl}.  We divide
  the proof in its various parts to make it easier to follow the
  technical arguments.

  \emph{Step I: linearization of the vector field on patches around
    the equilibrium point.} From Lemma~\ref{le:limit-jacobian}, define
  $\II_0 = \setdef{i \in \until{m}}{0 \in \mathrm{cl}(\DD_i)}$ and
  matrices $\{A_{0,i}\}_{i \in \II_0}$ as the limit points of the
  Jacobian matrices.  From the definition of piecewise $\CC^2$
  function, there exist $\CC^2$ functions
  $\{\map{f_i}{\DD_i^e}{\real^n}\}_{i \in \II_0}$ with $\DD_i^e$ open
  such that with $\mathrm{cl}(\DD_i) \subset \DD_i^e$ and the maps
  $f_{|\mathrm{cl}(\DD_i)}$ and $f_i$ take the same value over the set
  $\mathrm{cl}(\DD_i)$. Note that $0 \in \DD_i^e$ for every $i \in
  \II_0$.  By definition of the matrices $\{A_{0,i}\}_{i \in \II_0}$,
  we deduce that $Df_i(0) = A_{0,i}$ for each $i \in
  \II_0$. Therefore, there exists a neighborhood $\NN_0 \subset
  \real^n$ of the origin and a set of $\CC^2$ functions
  $\{\map{g_i}{\real^n}{\real^n}\}_{i \in \II_0}$ such that, for all $
  i \in \II_0$, $ f_i(x) = A_{0,i} x + g_i(x)$, for all $x \in \NN_0
  \cap \DD_i^e$, where
  \begin{align}\label{eq:g-cond}
    g_i(0) = 0 \quad \text{ and } \quad \frac{\partial g_i}{\partial
      x} (0) = 0.
  \end{align}
  Without loss of generality, select $\NN_0$ such that $\NN_0 \cap
  \DD_i$ is empty for every $i \not \in \II_0$. That is, $\cup_{i \in
    \II_0} (\NN_0 \cap \mathrm{cl}(\DD_i))$ contains a neighborhood of
  the origin. With the above construction, the vector field $f$ in a
  neighborhood around the origin is written as
  \begin{align}\label{eq:fsystem-patch}
    f(x) = f_i(x) = A_{0,x} x +g_i(x), \text{ for all } x \in
    \NN_0 \cap \mathrm{cl}(\DD_i), i \in \II_0,
  \end{align}
  where for each $i \in \II_0$, $g_i$ satisfies~\eqref{eq:g-cond}.

  \emph{Step II: change of coordinates.}  Subsequently, from
  hypothesis (i), there exists an orthogonal matrix $Q \in \real^{n
    \times n}$, defining an orthonormal transformation denoted by
  $\map{\TT_Q}{\real^n}{\real^n}$, $x \mapsto (u,v)$, that yields the
  new form of~\eqref{eq:fsystem-patch} as
  \begin{align}\label{eq:uv-coordinates}
    \begin{bmatrix}
      \dot u
      \\
      \dot v
    \end{bmatrix}
    =
    \begin{bmatrix}
      0 & 0
      \\
      0 & \tilde{A}_{0,i}
    \end{bmatrix}
    \begin{bmatrix}
      u \\ v 
    \end{bmatrix}
    + 
    \begin{bmatrix}
      \tilde{g}_{i,1}(u,v)
      \\
      \tilde{g}_{i,2} (u,v)
    \end{bmatrix}
    , \text{ for all } (u,v) \in \TT_Q(\NN_0 \cap \mathrm{cl}(\DD_i)),
    \, i \in \II_0,
  \end{align}
  where for each $i \in \II_0$, the matrix $\tilde{A}_{0,i}$ has
  eigenvalues with negative real parts (cf. hypothesis (ii)) and for
  each $i \in \II_0$ and $k \in \{1,2\}$ we have
  \begin{align}\label{eq:g-cond-2}
    \tilde{g}_{i,k}(0,0) = 0, \quad \frac{\partial
      \tilde{g}_{i,k}}{\partial u} (0,0) & = 0, \quad \text{ and }
    \quad \frac{\partial \tilde{g}_{i,k}}{\partial v} (0,0) = 0.
  \end{align}
  With a slight abuse of notation, denote the manifold of equilibrium
  points in the transformed coordinates by $\EE$ itself, i.e., $\EE =
  \TT_Q(\EE)$.  From~\eqref{eq:uv-coordinates}, we deduce that the
  tangent and the normal spaces to the equilibrium manifold $\EE$ at
  the origin are $\setdef{(u,v) \in \real^p \times \real^{n-p} }{v =
    0}$ and $\setdef{(u,v) \in \real^p \times \real^{n-p} }{u=0}$,
  respectively.  Due to this fact and since $\EE$ is a submanifold of
  $\real^n$, there exists a smooth function
  $\map{h}{\real^p}{\real^{n-p}}$ and a neighborhood $\UU \subset
  \TT_Q(\NN_0) \subset \real^n$ of the origin such that for any $(u,v)
  \in \UU$, $v = h(u)$ if and only if $(u,v) \in \EE \cap
  \UU$. Moreover,
  \begin{equation}\label{eq:h-cond}
    h(0) = 0 \text{ and } \frac{\partial h}{\partial u}(0) = 0.
  \end{equation}
  Now, consider the coordinate $w = v - h(u)$ to quantify the distance
  of a point $(u,v)$ from the set $\EE$ in the neighborhood $\UU$.  To
  conclude the proof, we focus on showing that there exists a
  neighborhood of the origin such that along a trajectory
  of~\eqref{eq:uv-coordinates} initialized in this neighborhood, we
  have $w(t) \to 0$ and $(u(t),h(u(t))) \in \UU$ at all $t \ge 0$.  In
  $(u,w)$-coordinates, over the set $\UU$, the
  system~\eqref{eq:uv-coordinates} reads as
  \begin{align}\label{eq:uw-coordinates}
    \begin{bmatrix}
      \dot u
      \\ 
      \dot w 
    \end{bmatrix}
    =
    \begin{bmatrix}
      0 & 0
      \\ 
      0 & \tilde{A}_{0,i} 
    \end{bmatrix}
    \begin{bmatrix}
      u
      \\ 
      w
    \end{bmatrix}
    +
    \begin{bmatrix}
      \bar{g}_{i,1}(u,w)
      \\
      \bar{g}_{i,2}(u,w)
    \end{bmatrix}
    , \text{ for } (u,w+h(u)) \in \UU \cap \TT_Q( \mathrm{cl}(\DD_i)
    ), \, i \in \II_0,
  \end{align}
  where $\bar{g}_{i,1}(u,w) = \tilde{g}_{i,1}(u,w+h(u))$ and
  $\bar{g}_{i,2}(u,w) = \tilde{A}_{0,i} h(u) +
  \tilde{g}_{i,2}(u,w+h(u)) - \frac{\partial h}{\partial
    u}(u)(\tilde{g}_{i,1}(u,w+h(u)))$.  Further, the equilibrium
  points $\EE \cap \UU$ in these coordinates are represented by the
  set of points $(u,0)$, where $u$ satisfies $(u,h(u)) \in \EE \cap
  \UU$.  These facts, along with the conditions on the first-order
  derivatives of $\tilde{g}_{i,1}$, $\tilde{g}_{i,2}$
  in~\eqref{eq:g-cond-2} and that of $h$ in~\eqref{eq:h-cond} yield
  \begin{align}\label{eq:g-cond-3}
    \bar{g}_{i,k}(u,0) = 0 \text{ and } \frac{\partial
      \bar{g}_{i,k}}{\partial w}(0,0) = 0,
  \end{align}
  for all $i \in \II_0$ and $k \in \{1,2\}$.  Note that the functions
  $\bar{g}_{i,1}$ and $\bar{g}_{i,2}$ are~$\CC^2$. This implies that,
  for small enough $\epsilon > 0$, we have $\norm{\bar{g}_{i,k}(u,w)}
  \le M_{i,k} \norm{w}$, for $k \in \{1,2\}$, $i \in \II_0$, and $(u,w)
  \in B_{\epsilon}(0)$, where the constants $\{M_{i,k}\}_{i \in \II_0,
    k\in \{1,2\}} \subset \realpositive$ can be made arbitrarily small
  by selecting smaller $\epsilon$.  Defining $M_\eps =
  \max\setdef{M_{i,k}}{i \in \II_0 , k \in \{1,2\}}$,
  \begin{align}\label{eq:g-cond3}
    \norm{\bar{g}_{i,k}(u,w)} \le M_\eps \norm{w}, \, \text{ for } k
    \in \{1,2\} \text{ and } i \in \II_0.
  \end{align}

  \emph{Step III: Lyapunov analysis.}  With the bounds above, we proceed
  to carry out the Lyapunov analysis for~\eqref{eq:uw-coordinates}.
  Using the matrix $P$ from assumption (iii), define   the candidate
  Lyapunov function $V:\real^{n-p} \to \realnonnegative$
  for~\eqref{eq:uw-coordinates} as $ V(w) = w^\top P w$ whose Lie
  derivative along~\eqref{eq:uw-coordinates} is
  \begin{align*}
    \Lie_{~\eqref{eq:uw-coordinates}} V(w) & = w^\top
    (\tilde{A}_{0,i}^\top P + P \tilde{A}_{0,i}) w + 2 w^\top P
    \bar{g}_{i,2}(u,w),
    \\
    & \qquad \text{ for } (u,w+h(u)) \in \UU \cap
    \TT_Q(\mathrm{cl}(\DD_i) ), \, i \in \II_0.
  \end{align*}
  By assumption (iii), there exists $\lambda >0$ such that $ w^\top
  (\tilde{A}_{0,i}^\top P + P \tilde{A}_{0,i}) w \le - \lambda
  \norm{w}^2$.  Pick $\eps$ such that $(u,w) \in B_{\eps}(0)$ implies
  $(u,h(u)+w) \in \UU$. Then, the above Lie derivative can be upper
  bounded as
  \begin{align*}
    \Lie_{~\eqref{eq:uw-coordinates}} V(w) \le - \lambda \norm{w}^2 + 2
    M_\eps \norm{P} \norm{w}^2 = - \beta_1 \norm{w}^2, \quad \text{for }
    (u,w) \in B_\eps(0),
  \end{align*}
  where $\beta_1 = \lambda - 2M_\eps$. Let $\eps$ small enough so that
  $\beta_1 >0$ and therefore $\Lie_{~\eqref{eq:uw-coordinates}} V(w)
  \le - \beta_1 \norm{w}^2 < 0$ for $w \not = 0$.  Now assume that
  there exists a trajectory $t \mapsto (u(t),w(t))$
  of~\eqref{eq:uw-coordinates} that satisfies $(u(t),w(t)) \in
  B_{\epsilon}(0)$ for all $t \ge 0$.  Then, using the following
  \begin{align*}
    \lambda_{\min}(P) \norm{w}^2 \le w^\top P w \le \lambda_{\max}(P)
    \norm{w}^2,
  \end{align*}
  we get $V(w(t)) \le e^{-\beta_1 t /\lambda_{\max}(P)} V(w(0))$ along
  this trajectory. Employing the same inequalities again, we get
  \begin{equation}\label{eq:w-bound}
    \norm{w(t)}  \le K \norm{w(0)}e^{-\beta_2t}, 
  \end{equation}
  where $K = \sqrt{\frac{\lambda_{\max}(P)}{\lambda_{\min}(P)}}$ and $
  \beta_2 = \frac{\beta_1}{2 \lambda_{\max}(P)} > 0$.  This proves
  that $w(t) \to 0$ exponentially for the considered
  trajectory. Finally, we show that there exists $\delta > 0$ such
  that all trajectories of~\eqref{eq:uw-coordinates} with initial
  condition $(u(0),w(0)) \in B_{\delta}(0)$ satisfy $(u(t),w(t)) \in
  B_{\epsilon}(0)$ for all $t \ge 0$ and hence, converge to $\EE$.
  From~\eqref{eq:uw-coordinates},~\eqref{eq:g-cond3}
  and~\eqref{eq:w-bound}, we have
  \begin{align}
    \norm{u(t)} 
    & \le \norm{u(0)} + \int_0^t M_\eps K e^{-\beta_2 s}\norm{w(0)} ds,
    \le \norm{u(0)} + \frac{M_\eps K}{\beta_2}\norm{w(0)}.
    \label{eq:u-bound}
  \end{align}
  By choosing $\epsilon$ small enough, $M_\eps$ can be made arbitrarily
  small and $\beta_2$ can be bounded away from the origin. With this,
  from~\eqref{eq:w-bound} and~\eqref{eq:u-bound}, one can select a
  small enough $\delta > 0$ such that $(u(0),w(0)) \in B_{\delta} (0)$
  imply $(u(t),w(t)) \in B_{\epsilon}(0)$ for all $t \ge 0$ and $w(t)
  \to 0$.  From this, we deduce that the trajectories staring in
  $B_{\delta}(0)$ converge to the set $\EE$ and the origin is
  stable. Since $\xo$ was arbitrary, we conclude local
  asymptotic stability of  $\EE$. Convergence to a point
  follows from the application of Lemma~\ref{le:convtopoint}.
\end{proof}

The next example illustrates the application of the above result to
conclude local convergence of trajectories to a point in the manifold
of equilibria.

\begin{example}\longthmtitle{Asymptotic stability of a manifold of
    equilibria for piecewise $\CC^2$ vector fields}\label{ex:patchy}
  {\rm Consider the system $\dot x = f(x)$, where
    $\map{f}{\real^3}{\real^3}$ is given by
    \begin{equation}\label{eq:patchy-vf-ex}
      f(x) = \begin{cases}
        \begin{bmatrix}
          -1 & 1 & 0 
          \\
          1 & -2 & 1 
          \\
          0
          & 1 & -1
        \end{bmatrix}
        \begin{bmatrix}
          x_1
          \\
          x_2
          \\
          x_3
        \end{bmatrix}
        + (x_1 - x_3)^2
        \begin{bmatrix}
          1
          \\
          1 
          \\
          1
        \end{bmatrix}
        ,
        & \text{ if } x_1 - x_3 \ge 0, 
        \\
        \begin{bmatrix}
          -2 & 1 & 1
          \\
          1 & -2 & 1
          \\
          1 & 1 & -2
        \end{bmatrix}
        \begin{bmatrix}
          x_1 
          \\
          x_2
          \\
          x_3
        \end{bmatrix}
        + (x_1 - x_3)^2 (1-x_1+x_3)
        \begin{bmatrix}
          1
          \\
          1 
          \\
          1
        \end{bmatrix}
        , & \text{ if } x_1 - x_3 < 0 .
      \end{cases}
    \end{equation} 
    The set of equilibria of $f$ is the one-dimensional manifold $\EE
    = \setdef{x \in \real^3}{x_1 = x_2 = x_3}$.  Consider the regions
    $\DD_1 = \setdef{x \in \real^2}{x_1 - x_3 >0}$ and $\DD_2 =
    \setdef{x \in \real^2}{x_1 - x_3 < 0}$. Note that $f$ is locally
    Lipschitz on $\real^3$ and $\CC^2$ on $\DD_1$ and $\DD_2$.  At any
    equilibrium point $\xo \in \EE$, the limit point of the
    generalized Jacobian belongs to $\{A_1,A_2\}$, where
    \begin{align*}
      A_1 = \begin{bmatrix} -1 & 1 & 0 \\ 1 & -2 & 1 \\ 0 & 1 &
        -1\end{bmatrix} \text{ and } A_2 =
              \begin{bmatrix} -2 & 1 & 1 \\ 1 & -2 & 1 \\ 1 & 1 & -2
              \end{bmatrix}.
    \end{align*}
    With the orthogonal matrix $Q = \left[
      \begin{matrix} 1 & 1 & 1 \\ 1 & -1 & 1 \\ 1 & 0 &
        -2 \end{matrix} \right]$ we get,
    \begin{align*}
      Q^\top A_1 Q = \begin{bmatrix} 0 & 0 & 0\\ 0 & -5 & 3 \\ 0 & 3 &
        -9 \end{bmatrix}, \qquad Q^\top A_2 Q = \begin{bmatrix} 0 & 0
        & 0\\ 0 & -6 & 0 \\ 0 & 0 & -18 \end{bmatrix}.
    \end{align*}
    The nonzero $2 \times 2$-submatrices obtained in the above
    equation have eigenvalues with negative real parts and have the
    identity matrix as a common Lyapunov function.  Therefore, from
    Proposition~\ref{pr:convtoman-cl}, we conclude that $\EE$ is
    locally asymptotically stable under $\dot x = f(x)$, as
    illustrated in Figure~\ref{fig:patchy-vf-ex}.
    \begin{figure}[htb!]
      \centering
      \subfigure[Trajectory]{\includegraphics[width=.44\linewidth]{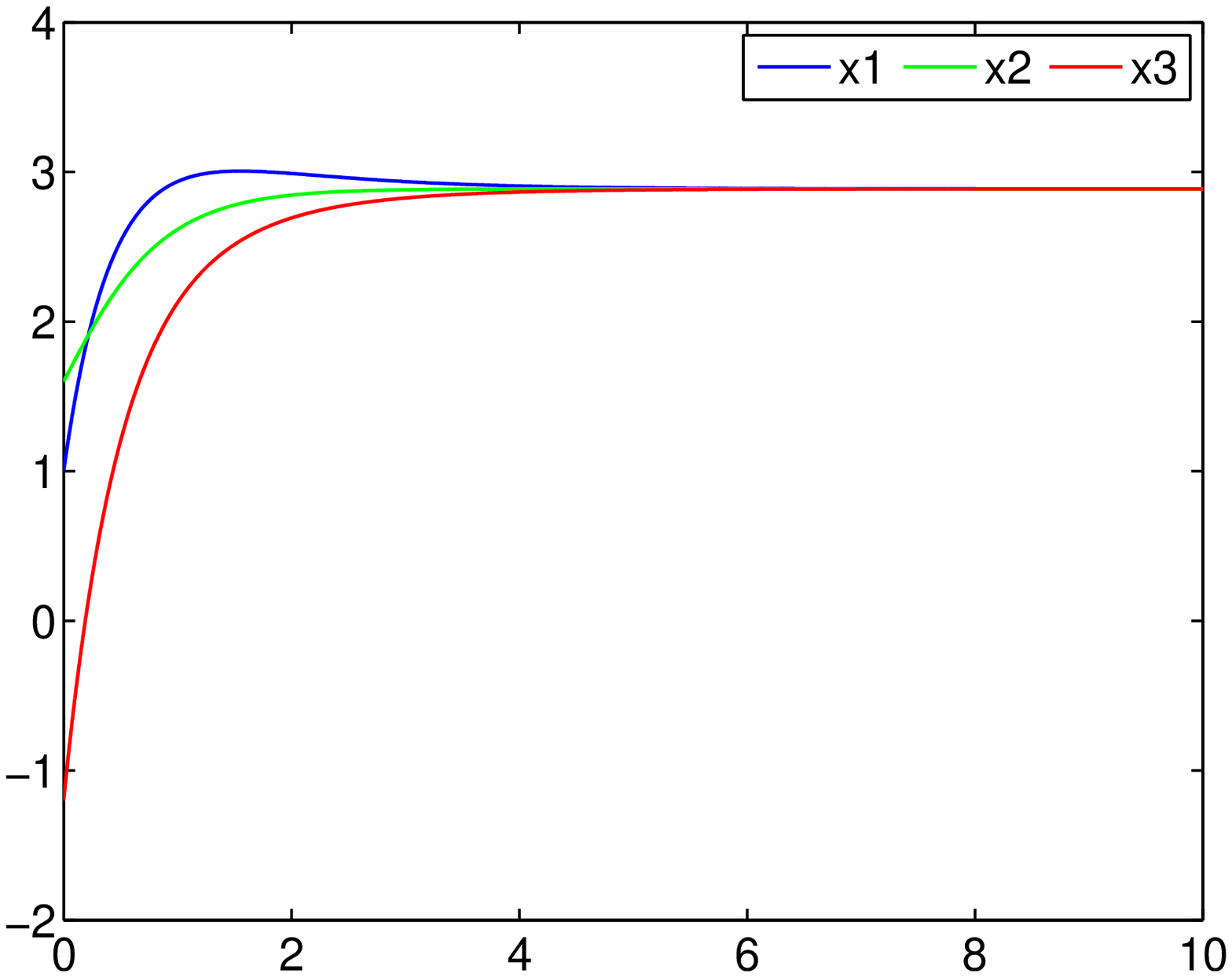}}
      \quad
      \subfigure[Distance to equilibrium set]{\includegraphics[width=.44\linewidth]{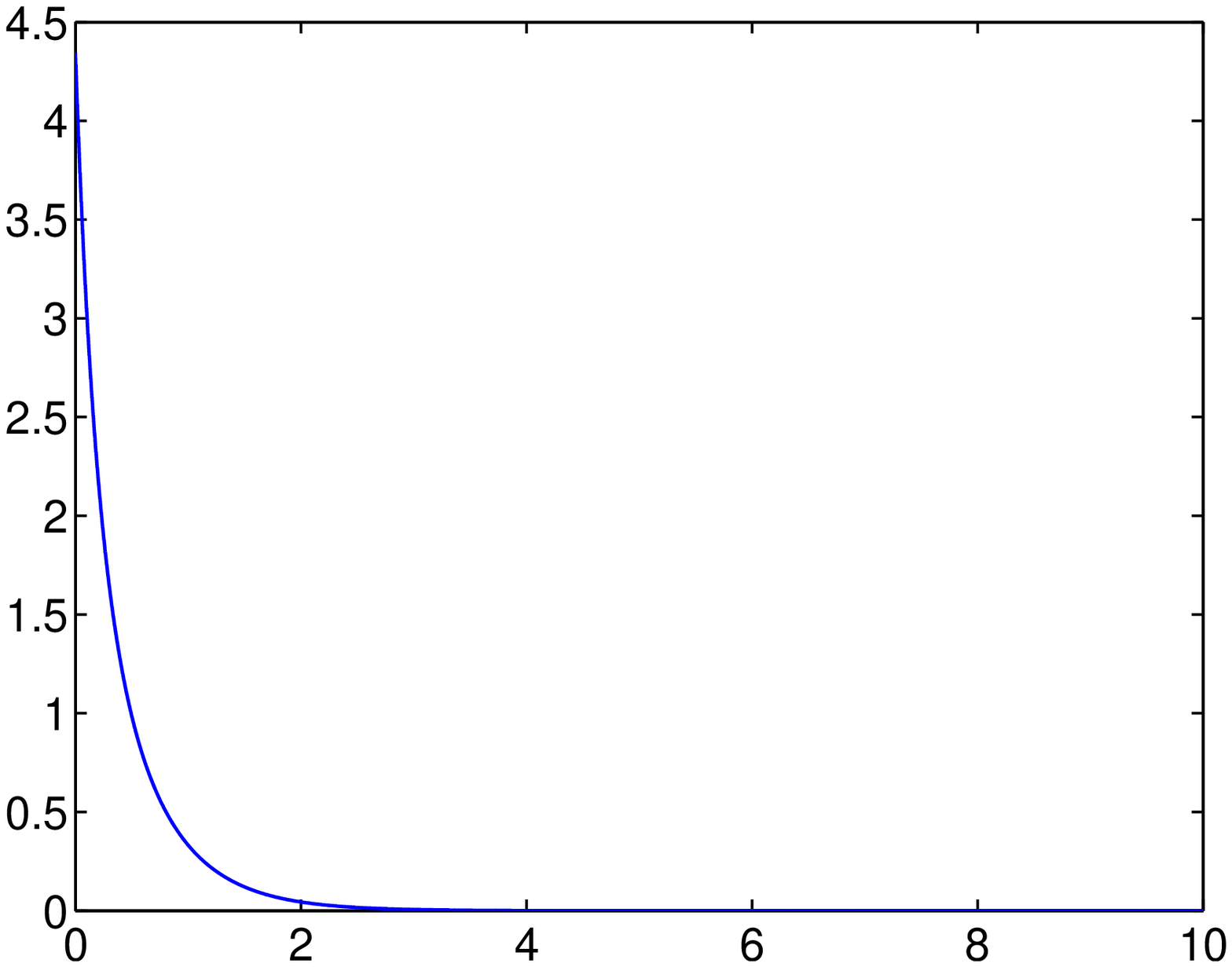}}
      \caption{(a) Trajectory of the vector field $f$ defined
        in~\eqref{eq:patchy-vf-ex}. The initial condition is $x =
        (1,1.6,-1.2)$. The trajectory converges to the equilibrium
        point $(2.88,2.88,2.88)$.  (b) Evolution of the distance to
        the equilibrium set $\EE$ of the trajectory.}
      \label{fig:patchy-vf-ex}
      \vspace*{-1ex}
    \end{figure}
    \oprocend }
\end{example}

\end{document}